\documentclass{amsart}

\usepackage[latin1]{inputenc}
\usepackage{amssymb}
\usepackage{amsmath}
\usepackage{enumerate}
\usepackage{epsfig}
\usepackage{subfigure}
\usepackage[all]{xy}
\usepackage{longtable}

\newtheorem{theorem}{Theorem}[section]

\newtheorem{proposition}[theorem]{Proposition}

\theoremstyle{definition}
\newtheorem{definition}[theorem]{Definition}
\newtheorem{example}[theorem]{Example}

\theoremstyle{remark}
\newtheorem{rem}[theorem]{Remark}
\theoremstyle{remark}

\numberwithin{equation}{section}
\newcommand{\cE}{\mathcal{E}}
\newcommand{\F}{\mathbb{F}}

\newcommand{\ra}{\rightarrow}

\newcommand{\diag}{\operatorname{diag}}
\newcommand{\PP}{\mathbb{P}}
\newcommand{\C}{\mathbb{C}}

\newcommand{\Z}{\mathbb{Z}}
\newcommand{\Q}{\mathbb{Q}}

\newcommand{\N}{\mathbb{N}}

\newcommand{\cO}{\mathcal{O}}

\newcommand{\id}{{\operatorname{id}}}
\newcommand{\Fix}{{\operatorname{Fix}}}
\newcommand{\Aut}{{\operatorname{Aut}}}
\newcommand{\Hirz}{{\mathbb{F}}}
\newcommand{\Pic}{\operatorname{Pic}}
\newcommand{\Cone}{\operatorname{Cone}}

\makeatletter
\def\blfootnote{\xdef\@thefnmark{}\@footnotetext}

\begin{document} 
\title{Calabi--Yau 3-folds of Borcea--Voisin type and elliptic fibrations}

\author{Andrea Cattaneo}
\address{Andrea Cattaneo, Dipartimento di Matematica, Universit\`a di Milano,
  via Saldini 50, I-20133 Milano, Italy}
\email{andrea.cattaneo1@unimi.it}

\author{Alice Garbagnati}
\address{Alice Garbagnati, Dipartimento di Matematica, Universit\`a di Milano,
  via Saldini 50, I-20133 Milano, Italy}
\email{alice.garbagnati@unimi.it}

\begin{abstract}
{We consider Calabi--Yau 3-folds of Borcea--Voisin type, i.e. Calabi--Yau 3-folds obtained as crepant resolutions of a quotient $(S\times E)/(\alpha_S\times \alpha_E)$, where $S$ is a K3 surface, $E$ is an elliptic curve, $\alpha_S\in \Aut (S)$ and $\alpha_E\in \Aut(E)$ act on the period of $S$ and $E$ respectively with order $n=2,3,4,6$. The case $n=2$ is very classical, the case $n=3$ was recently studied by Rohde, the other cases are less known.
First we construct explicitly a crepant resolution, $X$, of $(S\times E)/(\alpha_S\times \alpha_E)$ and we compute its Hodge numbers; some pairs of Hodge numbers we found are new. Then we discuss the presence of maximal automorphisms and of a point with maximal unipotent monodromy for the family of $X$. Finally, we describe the map $\mathcal{E}_n: X \ra S/\alpha_S$ whose generic fiber is isomorphic to $E$.
}
\end{abstract}

\maketitle

\section{Introduction}

In the Nineties several constructions of Calabi--Yau 3-folds, and in particular of mirror pairs of Calabi--Yau 3-folds, were proposed. One of them, the so called Borcea--Voisin construction, was independently described by Borcea and by Voisin in \cite{B}, \cite{V} respectively. The main idea is to consider a quotient of a product of Calabi--Yau varieties of lower dimension. More precisely, one considers two pairs: $(E, \iota_E)$, where $E$ is an elliptic curve and $\iota_E$ is its hyperelliptic involution, and $(S, \iota_S)$, where $S$ is a K3 suface (i.e.\ a Calabi--Yau surface) and $\iota_S$ is an involution of $S$ which does not preserve the period. The quotient $(S\times E)/(\iota_S\times \iota_E)$ is a singular 3-fold admitting a resolution which is a Calabi--Yau 3-fold. Both Borcea and Voisin explicitly constructed such a resolution and computed the Hodge numbers of the families of Calabi--Yau 3-folds obtained in such a way.
The classification of the involutions $\iota_S$ which act on a K3 surface without fixing the period of $S$ (and so of the K3 surfaces that can be used in the Borcea--Voisin construction) was given by Nikulin in \cite{N}. Several generalizations of the Borcea--Voisin construction have been introduced in the last years (see e.g. \cite{CH}, \cite{R}, \cite{R2}, \cite{D2}, \cite{G}), essentially considering desingularizations of quotients $(Y_1 \times Y_2)/(\alpha_1 \times \alpha_2)$, where $Y_i$ are Calabi--Yau varieties and $\alpha_i \in \Aut(Y_i)$. In order to obtain a Calabi--Yau variety one has to require that the automorphism $\alpha_i$ does not fix the period of $Y_i$, but $\alpha_1 \times \alpha_2$ fixes the wedge product of the periods of $Y_1$ and of $Y_2$.

Here we restrict our attention to Calabi--Yau 3-folds, and thus we can assume that $Y_1 =: S$ is a K3 surface and $Y_2 =: E$ is an elliptic curve. If we require that the order of $\alpha_1$ and $\alpha_2$ is the same, say $n$, then it has to be $n = 2, 3, 4, 6$ (for a more precise statement see Proposition \ref{prop: the construction}). Hence, we consider Calabi--Yau 3-folds constructed as resolution of a quotient $(S\times E) / (\Z/n\Z)$ for $n = 2, 3, 4, 6$ and we call them of Borcea--Voisin type.
In case $n = 2$ one obtains the ``classical'' and well known Borcea--Voisin construction.
A systematical analysis of the case $n = 3$ is presented in \cite{R} and \cite{D2} and uses the classification of the non-symplectic automorphisms of K3 surfaces of order 3, described independently by Artebani and Sarti, \cite{AS1}, and by Taki, \cite{T}.
Sporadic examples of the case $n = 4$ are analyzed in \cite{G}, where some peculiar K3 surfaces with a non-symplectic automorphism of order 4 are constructed and the associated Calabi--Yau 3-folds are presented. The complete classification of the K3 surfaces with non-symplectic automorphisms of order 4 and 6 is still unknown, but a lot of it is understood, see \cite{AS2} for $n = 4$ and \cite{D} for $n = 6$. Hence several families of Calabi--Yau 3-folds of Borcea--Voisin type obtained from quotients by automorphisms of order 4 and 6 can be described.

Given a quotient $(S\times E)/(\Z/n\Z)$ as before, there could exist more then one crepant resolution. We construct explicitly one specific crepant resolution (see Sections \ref{sec: construction X2}, \ref{sec: construction X3}, \ref{sec: construction X4}, \ref{subsec: X6}) and we call it of type $X_n$. The properties of the fixed loci of $\alpha_S^j$, $j=1,\ldots n-1$, on $S$ determine the Hodge numbers this 3-fold. We compute these for each admissible value of $n$ (see Propositions \ref{prop: hodge numbers X2}, \ref{prop: Hodge numbers X3}, \ref{prop: Borcea Voisin 4}, \ref{prop: hodge numbers 6}). Some of the Calabi--Yau 3-folds constructed have ``new'' Hodge numbers (here we refer to the database \cite{list} of the known Calabi--Yau 3-folds).

The 3-folds $X$ of type $X_n$ admit an automorphism induced by $\alpha_S \times \id$. In certain cases (see Proposition \ref{prop: h21=m-1 maximal automorphism}) this automorphism is a maximal automorphism for the family, i.e. it deforms to an automorphism of the varieties which are deformations of $X$. Such an automorphism acts non-trivially on the period of $X$. In \cite{R}, the maximal automorphisms which act non-trivially on the period are analyzed. In particular it is proved that if a family of Calabi--Yau 3-folds admits a maximal automorphism which acts on the period as the multiplication by an $n$-th root of unity, $n \neq 2$, then the family does not admit a point with maximal unipotent monodromy.
In \cite{R} the families of Calabi--Yau 3-folds of Borcea--Voisin type associated to $n = 3$ are considered and the ones with maximal automorphisms are classified, in order to construct families of Calabi--Yau 3-folds without maximal unipotent monodromy. Similarly, here we consider the family of Calabi--Yau 3-folds of Borcea--Voisin type associated to $n = 4, 6$ without maximal unipotent monodromy (see Remarks \ref{rem: no mum 4}, \ref{rem: maximal autom 6} and Tables \ref{table: Hodge numbers case 1}, \ref{table: only rational curves}, \ref{table: Hodge numbers X6}). In some peculiar cases the variation of Hodge structures of the family of Calabi--Yau 3-folds of type $X_n$ is essentially the variation of Hodge structures of a family of curves (see Remarks \ref{rem: variationa Hodge structure 3}, \ref{rem: isotrivial fib. S4, no mum}, and Example \ref{example no mum X6}).

By construction each variety $X$ of type $X_n$ is endowed with a map $\mathcal{E}_n: X\ra S/\alpha_S$ whose generic fiber is an elliptic curve isomorphic to $E$. The study of this map is one of the main tools of this paper (see Propositions \ref{prop: ell. fib. 2}, \ref{prop: ell. fib. 3}, \ref{prop: ell. fib. 4}, \ref{prop: ell. fib. 6}): $\mathcal{E}_n$ is an elliptic fibration (with section) if and only if $\alpha_S^j$ does not fix isolated points for any $j=1,\ldots n-1$. On the other hand, if $\alpha_S^j$ fixes some isolated point for a certain $j$, the fibers of $\mathcal{E}_n$ over the image of these points in $S/\alpha_S$ are the unique fibers which are not of Kodaira type, and they contain divisors. In any case a distinguished (rational) section is naturally given.
In case $S/\alpha_S$ (i.e. the base of the fibration) is smooth, we give a Weierstrass equation for $\mathcal{E}_n$ (see Sections \ref{Weierstrass 2}, \ref{Weierstrass 3-P1P1}, \ref{Weierstrass 3 F6}, \ref{Weierstrass 4 P2-proj}, \ref{Weierstrass 4 P2-proj}, \ref{Weeierstrass 6-P2 proj}, \ref{Weierstrass 6-F12}).

{\bf Acknowledgements:} The authors are grateful to Bert van Geemen for his support and his suggestions. They want to thank Michela Artebani and Alessandra Sarti for several useful discussions. The second author would like to thank the Fields Insituite and the organizers of the workshop ``Hodge theory in String Theory" (Fields Institute, November 2013) for giving her the possibility to present this work in a very stimulating atmosphere. 

The second author is supported by FIRB 2012 ``Moduli Spaces and Their Applications" and by PRIN 2010--2011 ``Geometria delle variet\`a algebriche".

\section{Preliminary results}

\subsection{Calabi--Yau $d$-folds and their automorphisms}\label{sect: cy and automorphism}

\begin{definition}
A \emph{Calabi--Yau $d$-fold} is a smooth compact K\"ahler variety of dimension $d$ such that the canonical bundle of $X$ is trivial and $h^{i, 0}(X) = 0$ if $i \neq 0, d$.

If $X$ is a Calabi--Yau $d$-fold, $H^{d, 0}(X) \simeq \C$ is generated by any non-trivial element in $H^{d, 0}(X)$. We call \emph{period} of $X$, denoted by $\omega_X$, a chosen non trivial element in $H^{d,0}(X)$, so $H^{d,0}(X)=\langle\omega_X\rangle$.

Let $\alpha$ be an automorphism of $X$. Then $\alpha^*$ acts on $H^d(X, \C)$ preserving the Hodge structure. In particular, $\alpha^*(\omega_X) = \lambda_\alpha \omega_X$ for a certain $\lambda_\alpha \in \C^*$. If $\lambda_\alpha = 1$, we will say that $\alpha$ preserves the period of $X$.
\end{definition}

The Calabi--Yau varieties of dimension 1 are the elliptic curves, the Calabi--Yau varieties in dimension 2 are the K3 surfaces.

\subsubsection{Automorphisms of elliptic curves}\label{subsec: automorphisms elliptic curve}

\begin{proposition}[{see e.g.\ \cite{ST}}]\label{prop: autom elliptic curves}
Let $E$ be an elliptic curve (i.e.\ a Calabi--Yau variety of dimension 1). If $\alpha_E$ is an automorphism of $E$ which does not preserve the period, then $\alpha^*(\omega_E) = \zeta_n \omega_E$ where $n = 2, 3, 4, 6$. If $n \neq 2$, then $E$ is a rigid elliptic curve with complex multiplication. More precisely: if $n = 3, 6$, then $j(E) = 0$; if $n = 4$, then $j(E) = 1728$.
\end{proposition}

We recall that any elliptic curve over $\C$ admits a Weierstrass equation of the form $v^2 = u^3 + au + b$. The hyperelliptic involution, denoted by $\iota$, acts on these coordinates in the following way: $\iota: (u, v) \mapsto (u, -v)$. It acts as the multiplication by $-1$ on the period.

We will denote by $E_{\zeta_3}$ the elliptic curve with $j$-invariant equal to 0, and with $E_i$ the one with $j$-invariant equal to $1728$.

We recall that a Weierstrass equation for $E_{\zeta_3}$ is $v^2 = u^3 + 1$. We denote by $\alpha_E: E_{\zeta_3} \ra E_{\zeta_3}$ the automorphism $\alpha_E: (u, v) \mapsto (\zeta_3 u, v)$ and by $\beta_E := \alpha_E^2 \circ \iota_E$.

Similarly, a Weierstrass equation for $E_{i}$ is $v^2 = u^3 + u$. We denote by $\alpha_E: E_{i} \ra E_{i}$ the automorphism $\alpha_E: (u, v) \mapsto (-u, iv)$ and we observe that $\alpha_E^2 = \iota$.

\subsubsection{Automorphisms of K3 surfaces}

\begin{definition}
Let $S$ be a K3 surface. An automorphism $\alpha_S \in \Aut(S)$ which preserves the period is called \emph{symplectic}. An automorphism $\alpha_S \in \Aut(S)$ of finite order $n := |\alpha_S|$ is \emph{purely non-symplectic} (of order $n$) if $\alpha_S^*(\omega_S) = \zeta_n \omega_S$, where $\zeta_n$ is a primitive $n$-th root of unity.
\end{definition}

We observe that the choice of the period of a K3 surface determines a symplectic structure on $S$ (this motivates the previous definition of symplectic automorphism). 
\begin{proposition}[{\cite{K}}]
Let $S$ be a K3 surface and $\alpha_S$ a purely non-symplectic automorphism of order $n$. Then $n \leq 66$ and if $n = p$ is a prime number, then $p \leq 19$. For every $p \leq 19$ there exists at least one K3 surface admitting a purely non-symplectic automorphism of order $p$.
\end{proposition}

In the following we will be interested in purely non-symplectic automorphisms of K3 surfaces of order 2, 3, 4, 6. So we recall that there is a complete classification of the K3 surfaces admitting a purely non-symplectic automorphism of prime order (\cite{AST}) and partial results on K3 surfaces admitting a purely non-symplectic automorphism of order 4 and 6 (\cite{AS2} and \cite{D} respectively).

\begin{proposition}[{see e.g. \cite{AST}}]
Let $\alpha_S$ be a purely non-symplectic automorphism of order $n$. Let $\Fix_{\alpha_S}(S) := \{s \in S | \alpha_S(s) = s \}$ be the fixed locus of $\alpha_S$. Then there are the following possibilities:
\begin{enumerate}
\item $\Fix_{\alpha_S}(S)$ is empty, in this case $n = 2$;
\item $\Fix_{\alpha_S}(S)$ is the disjoint union of two curves of genus 1, in this case $n = 2$;
\item $\Fix_{\alpha_S}(S) = C \coprod_{i = 1}^{k - 1} R_i \coprod_{j = 1}^{n} P_j$ where $P_j$ are isolated fixed points, $C$ and $R_j$ are curves, $C$ is the one with highest genus $g(C) \geq 0$ and $R_j$ are rational curves.
\end{enumerate}
\end{proposition}

In the third case the fixed locus of $\alpha_S$ is determined by the triple $(g(C), k, n)$ and its Euler characteristic is $e(Fix_{\alpha_S}(S))=n+2k-2g(C)$. If $n = p$ is a prime number then for every prime $p \leq 19$ there exists a known finite list of admissible triples $(g(C), k, n)$ such that there exists at least a K3 surface admitting a non-symplectic automorphism of order $p$ with fixed locus associated to one of these triples.
 
\begin{definition}\label{notation K3 surfaces}
Let $\alpha_S$ be a purely non-symplectic automorphism of order $n$ on a K3 surface $S$. Let us denote by $H^2(S, \C)_{\zeta_n^j}$ the eigenspace of the eigenvalue $\zeta_n^j$ for the action of $\alpha_S^*$ on $H^2(S, \C)$. For every $i \in \Z / n\Z$ of order $n$, the dimension $\dim(H^2(S, \C)_{\zeta_n^i})$ does not depend on $i$ and will be denoted by $m$. We will denote by $r := \dim(H^2(S, \C)^{\alpha_S})$.
\end{definition}

\begin{proposition}\label{prop: dimension family K3} Let $S$ be a K3 surface admitting a purely non-symplectic automorphism $\alpha_S$ of order $n$. The numbers $r$ and $m$ are uniquely determined by the Euler characteristics of the fixed loci of $\alpha_S^j$ for $j=1,\ldots n-1$.

The dimension of the families of K3 surfaces $S$ admitting a purely non symplectic automorphism $\alpha_S$ of order $n$ with prescribed Euler characteristics of the fixed loci of $\alpha_S^j$, $j=1,\ldots n-1$, is $m-1$.\end{proposition}
\proof
The first statement follows immediately by the Lefschetz fixed points formula, the second one by \cite[Section 11]{DK}. \endproof

We observe that if $\alpha_S$ acts as $\zeta_n$ on the period of $S$, then $\dim(H^{1, 1}(S)_{\zeta_n}) = m - 1$.

\subsubsection{Maximal automorphisms of Calabi--Yau 3-folds}
\begin{definition}
Let $\mathcal{X} \ra B$ be a family of Calabi--Yau $d$-folds $X_t$. A \emph{maximal automorphism} of such a family of Calabi--Yau $d$-folds is an automorphism $\alpha_{\overline{t}}$ of a smooth fiber $X_{\overline{t}}$ which extend to the local universal deformation of $X_{\overline{t}}$.
\end{definition}

Let $B$ be a polydisc and let $\mathcal{X} \ra B$ be a local family of Calabi--Yau 3-folds. Let us fix $\overline{t}\in B$. For a generic $t \in B$, $H^3(X_t, \Q) \simeq H^3(X_{\overline{t}}, \Q)$. If $\alpha_{\overline{t}}$ is an automorphism of $X_{\overline{t}}$, then $\alpha_{\overline{t}}^*$ acts on $H^3(X_t, \Q)$ for any $t$. If $\alpha_{\overline{t}}$ is a maximal automorphism for the family, the action of $\alpha_{\overline{t}}^*$ on $H^3(X_t, \Q)$ is induced by the automorphism $\alpha_t$ of $X_t$. In particular the action of $\alpha_t^*$ preserves the Hodge structure of $H^3(X_t,\Q)$, i.e. $\alpha_t^*=\alpha_{\overline{t}}^*$  is compatible with the variation of the Hodge structures of $X_{\overline{t}}$. Thus, the action of $\alpha_{t}^*$ on $\omega_{X_t}$ does not depend on $t$, i.e.\ there exists a complex non-trivial number $\lambda$ such that $\alpha^*_{t}(\omega_{X_t}) = \lambda(\omega_{X_t})$ for every generic $t\in B$. 

\begin{proposition}[{\cite[Theorem 8]{R}}]\label{prop: maximal automorphism, order}
Let $X_{\overline{t}}$ be a  Calabi--Yau and let $\alpha_{\overline{t}}\in \Aut(X_{\overline{t}})$ be a maximal automorphism of the family of $X_{\overline{t}}$. If $\alpha_{\overline{t}}$ acts on the period of $X_{\overline{t}}$ with finite order then either it acts trivially or with one of the following orders: 2,3,4,6 (i.e.\ the value $\lambda$ associated to $\alpha_{\overline{t}}$ takes one of the following values $1, -1, \zeta_3, i, \zeta_6$).
\end{proposition}

\begin{proposition}[{\cite[Theorem 7]{R}}]\label{prop: no unipotent monodromy}
Let $X_{\overline{t}}$ be a  Calabi--Yau and let $\alpha_{\overline{t}}\in \Aut(X_{\overline{t}})$ be a maximal automorphism of the family of $X_{\overline{t}}$. 
If $\alpha_{\overline{t}}$ acts with order 3, 4 or 6 on the period, then the family of $X_{\overline{t}}$ does not admit a point with maximal unipotent monodromy.
\end{proposition}

The previous Proposition is used in several papers (see e.g.\ \cite{R}, \cite{GvG}, \cite{G}) to construct explicit examples of Calabi--Yau 3-folds without maximal unipotent monodromy. We observe that there exist also families of Calabi--Yau 3-folds without maximal unipotent monodromy which do not admit a maximal automorphism acting on the period as described in Proposition \ref{prop: no unipotent monodromy}, see \cite{CvD}.

\subsection{Elliptic fibrations on 3-folds}\label{sec: elliptic fibrations on 3-folds}

Let $Y$ be a 3-fold. We now define the notion of elliptic fibration. Since several 3-folds we construct in the following do not admit an elliptic fibration, but have a natural map whose generic fiber is an elliptic curve, we also give a less restrictive definition (the one of almost elliptic fibration) which is useful to describe our situation.

\begin{definition}
Let $Y$ be a 3-fold and $R$ be a surface. We will say that a surjective map with connected fibers $\cE: Y \ra R$ is an \emph{elliptic fibration} if:
\begin{enumerate}
\item the generic fiber of $\cE$ is a smooth genus one curve;
\item a section of $\cE$ is given, i.e.\ there exists a map $s: R \ra Y$ such that $\cE \circ s = \id$;
\item all the fibers of $\cE$ are of dimension 1.
\end{enumerate}
\end{definition}


\begin{definition}
Let $Y$ be a 3-fold and $R$ be a (possibly singular) surface. We will say that a surjective map with connected fibers $\pi: Y \ra R$ is an \emph{almost elliptic fibration} if:
\begin{enumerate}
\item the generic fiber of $\pi$ is a smooth genus one curve;
\item a rational section is given, i.e.\ there exists a rational map $s: R \dashrightarrow Y$ such that $\pi \circ s = \id$ on the domain of $s$.
\end{enumerate}
\end{definition}

\begin{rem}
We observe that if $\pi: Y \ra R$ is an almost elliptic fibration, then we do not require that all the fibers have dimension 1, and in fact we accept the presence of divisors contained in a fiber.
\end{rem}

\subsection{Hirzebruch surfaces}\label{sec: Hirzebruch}
In the following we will construct elliptic fibrations over Calabi--Yau 3-folds whose bases are one of the following surfaces: $\PP^2$, $\PP^1 \times \PP^1$ or $\F_n$, the Hirzebruch surface. For this reason we recall some results on the Hirzebruch surfaces.

Let $\F_n$ denote the Hirzebruch surface $\F_n := \PP(\cO_{\PP^1}(-n) \oplus \cO_{\PP^1})$, $n\in\N$. These surfaces are toric varieties, whose fan has four edges (\cite{F})
\[\begin{array}{ll}
v_s = (1, 0), & v_t = (-1, n),\\
v_y = (0, -1), & v_z = (0, 1),
\end{array}\]
and four maximal cones
\[\begin{array}{ll}
\Cone(v_s, v_z), & \Cone(v_z, v_t),\\
\Cone(v_t, v_y), & \Cone(v_y, v_s).
\end{array}\]
We can describe $\F_n$ also as a quotient space (\cite[$\S$5]{CLS}): it is the quotient of $\C^4_{(s, t, y, z)} \smallsetminus \{ s = t = 0, y = z = 0 \}$ by the action of $\C^* \times \C^*$
\[(\lambda, \mu)(s, t, y, z) = (\lambda s, \lambda t, \lambda^n \mu y, \mu z)\]
and so we can use $(s: t: y: z)$ as global homogeneous toric coordinates on $\F_n$.\\
From the fan we can also see that the Picard group of $\F_n$ is generated by the four divisors $D_s$, $D_t$, $D_y$ and $D_z$, with the relations
\[D_s \equiv D_t, \qquad D_y \equiv n D_t + D_z,\]
and so
\[\Pic \F_n = \Z \cdot D_t \oplus \Z \cdot D_z.\]
The intersection properties of these divisors are $D_t^2 = 0$, $D_z^2 = -n$ and $D_t D_z = 1$.
We observe that $D_y^2 = n$ and so we call $D_y$ the positive curve, and $D_z$ the negative curve.

We recall $K_{\F_{n}} = -(n + 2) D_t - 2 D_z$.

\begin{rem}\label{rem: autom on Hirz}
Every Hirzebruch surface admits an automorphism of order $d$, for every $d \in \N$, whose quotient is another Hirzebruch surface. Indeed let us consider
\[\alpha: (s: t: y: z) \longmapsto (s: t: \zeta_d y: z),\]
where $\zeta_d$ denotes a $d$-th primitive root of unity. Then $\alpha$ is an automorphism of order $d$ on $\F_n$, whose fixed locus consists of the two disjoint rational curves $y = 0$ and $z=0$ respectively.

The quotient of $\F_n$ by the action of this automorphism is another Hirzebruch surface:
\[\begin{array}{cccc}
q: & \F_n & \longrightarrow & \F_{dn}\\
   & (s: t: y: z) & \longmapsto & (s: t: y^d: z^d).
\end{array}\]
In particular, $\F_{2n}$ is the quotient of $\F_n$ by the involution $(s: t: y: z) \longmapsto (s: t: -y: z)$.
\end{rem}

\section{Calabi--Yau 3-folds of Borcea--Voisin type}

Here we describe the main construction of this paper and we summarize some of our main results (see Propositions \ref{prop: Hodge numbers of CY of BV type}, \ref{prop: h21=m-1 maximal automorphism}, \ref{prop: existences of BV maximal}). 

At least in the cases $n = 2$ and $n = 3$ the construction is well known, and in particular it was introduced in case $n = 2$ by Borcea and by Voisin, see \cite{B} and \cite{V} respectively. In case $n = 3$ it is extensively studied by Rohde in \cite{R}. Several generalizations of such construction are proposed, see for example in \cite{D2} and \cite{G}. Here we describe one of them.

First we recall an essential result on the existence of certain crepant resolutions. Let $Z$ be a 3-fold with trivial canonical bundle and let $G$ be a finite group of automorphisms of $Z$ which preserves the period. This implies that for every $g\in G$, the action of $g^*$ on the tangent space at a fixed point is represented by a diagonal matrix in $SL(3)$. Under this condition the following holds:

\begin{proposition}[\cite{Y} \cite{Bat2}]\label{prop: existence crepant resolution}
The singular 3-fold $Z/G$ admits a crepant resolution, and the Hodge numbers of any such a resolution do not depend on the crepant resolution we are considering.
\end{proposition}

\begin{proposition}\label{prop: the construction}
Let $S$ be a K3 surface admitting a purely non-symplectic automorphism $\alpha_S$ of order $n$ such that $\alpha^*_S(\omega_S) = \zeta_n \omega_S$. Let $E$ be an elliptic curve admitting an automorphism $\alpha_E$ such that  $\alpha_E^*(\omega_E) = \zeta_n \omega_E$. Then $n = 2, 3, 4, 6$ and $(S \times E)/(\alpha_S \times \alpha_E^{n - 1})$ is a singular variety which admits a desingularization which is a Calabi--Yau 3-fold.
\end{proposition}
\begin{proof}
The condition on $n$ follows by Proposition \ref{prop: autom elliptic curves}.\\
The 3-fold $S \times E$ has trivial canonical bundle and a generator of $H^{3, 0}(S \times E, \C)$ is $\omega_S \wedge \omega_E$. By construction $\alpha_S \times \alpha_E^{n - 1}$ preserves the period, hence there exists a crepant resolution of $(S \times E) / (\alpha_S \times \alpha_E^{n - 1})$, i.e.\ a resolution with trivial canonical bundle, by Proposition \ref{prop: existence crepant resolution}. Since $H^{2, 0}(S \times E) = \langle \omega_S \rangle$ and $H^{1, 0}(S \times E) = \langle \omega_E \rangle$ are not preserved by $\alpha_S \times \alpha_E^{n-1}$, and since $h^{i, 0}$ are birational invariant for any $i$, the Hodge numbers $h^{1, 0}$ and $h^{2, 0}$ of any resolution of $(S \times E) / (\alpha_S \times \alpha_E^{n - 1})$ are trivial. Hence there exists a resolution of $(S \times E) / (\alpha_S \times \alpha_E^{n - 1})$ which is a Calabi--Yau 3-fold.
\end{proof}

\begin{definition}
Any crepant resolution of $(S \times E) / (\alpha_S \times \alpha_E^{n - 1})$ will be called a \emph{Calabi--Yau 3-fold of Borcea--Voisin type}.
\end{definition}

\begin{proposition}\label{prop: Hodge numbers of CY of BV type} The Hodge numbers of any Calabi--Yau 3-fold of Borcea--Voisin type depend only on the topological properties of the fixed loci of $\alpha_S^j$, for $j=1,\ldots, n-1$. \end{proposition}
This Proposition follows immediately by the computations of the Hodge numbers of $X$ done in Propositions \ref{prop: hodge numbers X2}, \ref{prop: Hodge numbers X3}, \ref{prop: Borcea Voisin 4}, \ref{prop: hodge numbers 6} for $n=2,3,4,6$ respectively. Anyway it is based on a general idea:\\
Let $\widetilde{(S \times E) / (\alpha_S \times \alpha_E^{n - 1})}$ be a crepant resolution of $(S \times E) / (\alpha_S \times \alpha_E^{n - 1})$. Since $\widetilde{(S \times E) / (\alpha_S \times \alpha_E^{n - 1})}$ is a Calabi--Yau 3-fold, $h^{0, 0} = h^{3, 0} = 1$ and $h^{1, 0} = h^{2, 0} = 0$. The numbers $h^{1, 1}$ and $h^{2, 1}$ depend on the action of $\alpha_S \times \alpha_E^{n - 1}$ on $S \times E$. They are the sum of two contributions: one comes from the desingularization of the singular locus of $(S \times E) / (\alpha_S \times \alpha_E^{n - 1})$, the other comes from the cohomology of $S \times E$ which is invariant for $\alpha_S \times \alpha_E^{n - 1}$. Since the fixed loci of $\alpha_E^j$, $j=1,\ldots n-1$ are uniquely determined by $n$, the fixed loci of $(\alpha_S\times\alpha_E^{n-1})^j$ depends only the properties of the fixed loci of $\alpha_S^j$. The part of the cohomology which comes from the cohomology of $S\times E$ can be computed in this general setting:
\begin{align}\label{formula: invariant h11 h21}\begin{array}{rl}
H^{1, 1}(S \times E)^{\alpha_S \times \alpha_E^{n - 1}} = & (H^{0, 0}(S) \otimes H^{1, 1}(E))^{\alpha_S \times \alpha_E^{n - 1}}\oplus (H^{1, 1}(S) \otimes H^{0, 0}(E))^{\alpha_S \times \alpha_E^{n - 1}}\\
= & (H^{0, 0}(S) \otimes H^{1, 1}(E)) \oplus \left(H^{1, 1}(S)^{\alpha_S} \otimes H^{0, 0}(E)\right);\\
H^{2, 1}(S \times E)^{\alpha_S \times \alpha_E^{n - 1}} = & (H^{2, 0}(S) \otimes H^{0, 1}(E))^{\alpha_S \times \alpha_E^{n - 1}} \oplus (H^{1, 1}(S) \otimes H^{1, 0}(E))^{\alpha_S \times \alpha_E^{n - 1}}\\
= & H^{1, 1}(S)_{\zeta_n} \otimes H^{1, 0}(E).
\end{array}\end{align}
With the notation introduced in Definition \ref{notation K3 surfaces}, the dimension of these spaces are $1 + r$ and $m - 1$ respectively if $n \neq 2$, and $1 + r$ and $1 + (m - 1) = m$ if $n = 2$. We recall $r$ and $m$ depends only on the properties of the fixed loci of $\alpha_S^j$, by Proposition \ref{prop: dimension family K3}.

\begin{definition}\label{defi: alphaX}
Let us consider the automorphism $\alpha_S \times \id \in \Aut(S \times E)$. It clearly commutes with $\alpha_S \times \alpha_E^{n - 1}$ and so descends to an automorphism $\alpha_X$ of any Calabi--Yau 3-fold $X$ which desingularizes $(S \times E) / (\alpha_S \times \alpha_E^{n - 1})$.
\end{definition}
\subsection{Borcea--Voisin maximal families}
By choosing a K3 surface $S$ with a non-symplectic automorphism $\alpha_S$ and an elliptic curve with an automorphism $\alpha_E$ as in Propostion \ref{prop: the construction}, we produce a Calabi--Yau 3-fold, $X$, of Borcea--Voisin type. We now consider the family of Calabi--Yau 3-folds which deform $X$. By the Tian--Todorov theorem, the dimension of such a family is $h^{2, 1}(X)$. In general not all the members of this family are of Borcea--Voisin type.

\begin{definition}
A \emph{Borcea--Voisin maximal family} is a family of Calabi--Yau 3-folds such that the generic member of this family is of Borcea--Voisin type.
\end{definition}

\begin{proposition}\label{prop: h21=m-1 maximal automorphism}
Let $n = |\alpha_S|$ and let $\mathcal{F}_X$ be the family of $X$.\begin{enumerate} \item If $n = 2$ and $h^{2, 1}(X) = m$ ($m$ as in Definition \ref{notation K3 surfaces}), then $\mathcal{F}_X$ is a Borcea--Voisin maximal family. If $n = 3, 4, 6$  and $h^{2, 1}(X) = m - 1$, then $\mathcal{F}_X$ is Borcea--Voisin maximal. \item If $\mathcal{F}_X$ is Borcea--Voisin maximal, then $\alpha_X$ is a maximal automorphism and $\alpha_X(\omega_X)=\zeta_n\omega_X$.\item If $\mathcal{F}_X$ is Borcea--Voisin maximal and $n=3,4,6$, then $\mathcal{F}_X$ does not admit maximal unipotent monodromy. \item   If $\mathcal{F}_X$ is Borcea--Voisin maximal and $n=3,4,6$, then the variation of the Hodge structures of $\mathcal{F}_X$ depends only on the variation of the Hodge structures of the family of $S$. \end{enumerate}
\end{proposition}
\begin{proof}
The dimension of the family of K3 surfaces $S$ admitting the purely non-symplectic automorphism $\alpha_S$ of order $n$ is $m - 1$ (see Proposition \ref{prop: dimension family K3}). The dimension of the family of elliptic curves $E$ admitting the automorphism $\alpha_E$ as considered is 1 if $n = 2$, and 0 otherwise. Let us consider the family of 3-folds which deforms $S\times E$. Its dimension is $m = (m - 1) + 1$ if $n = 2$ and $m - 1$ otherwise. The dimension of the family $\mathcal{F}_X$ is $h^{2, 1}(X)$. So, if $h^{2, 1}(X) = m$ for $n = 2$ or $h^{2, 1}(X) = m - 1$ for $n = 3, 4, 6$, these two families have the same dimension, and thus the generic deformation of $X$ is a crepant resolution of $(S \times E) / (\alpha_S \times \alpha_E^{n - 1})$. In particular, we obtain Borcea--Voisin maximal families and the automorphism $\alpha_X$ is maximal (since it is defined on every deformation of $(S\times E)/(\alpha_S\times\alpha_E^{n-1})$). By the Proposition \ref{prop: no unipotent monodromy}, if $\mathcal{F}_
X$ is maximal Borcea--Voisin and $n=3,4,6$, $\mathcal{F}_X$ does not admit maximal unipotent monodromy. If $n=3,4,6$, then $E$ is a rigid curve. Moreover, if $\mathcal{F}_X$ is a Borcea--Voisin maximal family, $H^{3,0}(X)\oplus H^{2,1}(X)=\left(H^{2,0}(S)\oplus H^{1, 1}(S)_{\zeta_n}\right) \otimes H^{1, 0}(E)=\left(H^{2,0}(S)\oplus H^{1, 1}(S)_{\zeta_n}\right) \otimes \C$ hence the variation of the Hodge structures of $\mathcal{F}_X$ depends only on the variation of the Hodge structures of $S$.
\end{proof}

Under certain hypothesis, if $S$ is a K3 surfaces with a purely non--symplectic automorphism, the variation of the Hodge structures of $S$ is essentially the variation of the Hodge structures of a family of curves.  So, if we are in case $(4)$ of the previous proposition and moreover the variation of the Hodge structures of $S$ depends only on the variation of the Hodge structures of a family of curves, then the variation of the Hodge structures of $X$ is essentially the variation of the Hodge structures of a family of curves. In particular the Picard--Fuchs equation of $X$ is the Picard--Fuchs equation of a family of curves. Examples of this phenomenon are given in \cite{GvG}, \cite{G}, in Remarks \ref{rem: variationa Hodge structure 3}, \ref{rem: no mum 4}, and in Section \ref{sec: example order 6 no mum}.
\begin{rem}\label{rem: other maximal automorphisms}
Let us assume that the family of $X$ is a Borcea--Voisin maximal family.  We recall  that $X$ is a desingularization of $(S\times E)/(\alpha_S\times\alpha_E)$ where $S$ lies in the family of the K3 surfaces admitting a non-symplectic automorphism $\alpha_S$ with certain properties. Let us assume that every K3 surface in the family of $S$ admits an automorphism $\sigma$ which commutes with $\alpha_S$. Then $\sigma \times \id \in \Aut(S \times E)$ induces a maximal automorphism of the family of $X$ (which is in general not equal to $\alpha_X$).
\end{rem}

\begin{proposition}\label{prop: existences of BV maximal} Let $X$ be a Calabi--Yau 3-fold of Borcea--Voisin type and $\mathcal{F}_X$ its family. 
If $\alpha_S^j\in \Aut(S)$ does not fix curves of positive genus for every $j=1,\ldots n-1$, then $\mathcal{F}_X$ is a Borcea--Voisin maximal family. 

There exists at least one Borcea--Voisin maximal family for every $n=2,3,4,6$.\end{proposition}
\proof The first statement follows by Proposition \ref{prop: h21=m-1 maximal automorphism} and by the computations of the Hodge numbers of $X$ in  Propositions \ref{prop: hodge numbers X2}, \ref{prop: Hodge numbers X3}, \ref{prop: Borcea Voisin 4}, \ref{prop: hodge numbers 6} for $n=2,3,4,6$ respectively. It is known that there exist purely non-symplectic automorphisms $\alpha_S$ of order 2,3 and 4 such that $\alpha_S^j$ does not fix curves of positive genus for every $j=1,\ldots n-1$ (see \cite{N}, \cite{AS1},\cite{AS2} and \cite{G} respectively; for $n=4$ see also  Table \ref{table: Hodge numbers case 1}, \ref{table: only rational curves} and Remark \ref{rem: no mum 4}). In case $n=6$ we construct an explicit example in Section \ref{sec: example order 6 no mum}.
\endproof

\subsection{Fibrations on Calabi--Yau 3-folds of Borcea--Voisin type}
In the next sections we will construct explicitly one crepant resolution $X$ of $(S\times E)/(\alpha_S\times \alpha_E^{n-1})$ and so a particular Calabi--Yau 3-fold of Borcea--Voisin type, called of type $X_n$. In the following proposition we summarize some geometric properties of these 3-folds. 

\begin{proposition}\label{prop: fibrations on Xn} Let $\mathcal{G}_n:X\ra E/\alpha_E^{n-1}\simeq\mathbb{P}^1$ and $\mathcal{E}_n:X\ra S/\alpha_S$ be the maps induced by $(S\times E)/(\alpha_S\times\alpha_E^{n-1})\ra E/\alpha_E^{n-1}$ and  $(S\times E)/(\alpha_S\times\alpha_E^{n-1})\ra S/\alpha_S$ respectively on $X$. Let $g:E\ra E/\alpha_E^{n-1}$ and $q:S\ra S/\alpha_S$ be the quotient maps.

The map $\mathcal{G}_n$ is an isotrivial K3 surfaces fibration. The fiber $\mathcal{G}_n^{-1}(g(P))$ is reducible if and only if $P\in E$ is a point with non trivial stabilizer for the action of $\alpha_E$ on $E$.

The map $\mathcal{E}_n$ is an almost elliptic fibration. More precisely:\\ 
$\bullet$ the fiber $\mathcal{E}_n^{-1}(q(Q))$ is isomorphic to $E$ if and only if $Q\in S$ has trivial stabilizer for the action of $\alpha_S$ on $S$;\\
$\bullet$ the fiber $\mathcal{E}_n^{-1}(q(Q))$ is singular of dimension 1 if and only if $Q\in S$ has a non trivial stabilizer for the action of $\alpha_S$ on $S$, but $Q$ is not an isolated fixed point for $\alpha_S^j$ for any $j=1,\ldots n-1$; in this case $\mathcal{E}_n^{-1}(q(Q))$ is of Kodaira type;\\
$\bullet$ the fiber $\mathcal{E}_n^{-1}(q(Q))$ contains divisors if and only if $Q\in S$ is an isolated fixed point for $\alpha_S^j$ for at least one $j\in\{1,\ldots n-1\}$.\end{proposition}
\proof The fibers of $\mathcal{F}_n$ are clearly equidimensional and the smooth ones are isomorphic to $S$. The ones which are not smooth contain divisors which come from the resolution of $(S\times E)/(\alpha_S\times\alpha_E^{n-1})$ and hence project to points in $E/\alpha^{n-1}$ which are branch points.

The map $\mathcal{E}_n$ will be considered in the Propositions \ref{prop: ell. fib. 2}, \ref{prop: ell. fib. 3}, \ref{prop: ell. fib. 4}, \ref{prop: ell. fib. 6}.\endproof

We observe that $\mathcal{E}_n$ is an elliptic fibration if and only if $\alpha_S^j$ does not fix isolated points for every $j\in\{1,\ldots n-1\}$. For example, this surely happens if $n=2$.

\section{The ``original'' Borcea--Voisin construction: order 2}

In this section we assume $|\alpha_S| = |\alpha_E| = 2$.  We present the explicit construction of a crepant resolution, of type $X_2$, of $(S \times E) / (\alpha_S \times \alpha_E)$ following \cite[Section 1]{V}.

\subsection{The construction of Calabi--Yau 3-folds of type $X_2$}\label{example: BorceaVoisin}\label{sec: construction X2}

Let $S$ be a K3 surface with a non-symplectic involution $\iota_S$. These K3 surfaces are classified in \cite{N}. The fixed locus of $\iota_S$ on $S$ is either empty or consists of $N$ curves, and $\iota_S$ linearizes near the fixed locus to the matrix $\diag(-1, 1)$. Let $E$ be an elliptic curve and $\iota_E$ its hyperelliptic involution. The fixed locus of $\iota_E$ consists of 4 points, $P_1$, $P_2$, $P_3$ and $P_4$, and clearly the local action of $\iota_E$ near the fixed points is $-1$.
Thus, the fixed locus of $\iota := \iota_S \times \iota_E$ on $S \times E$ consists of $4 N$ curves and the local action near the fixed locus linearizes to $\diag(-1, 1, -1)$. The singularities of $(S \times E) / \iota$ are the images of the curves in $S \times E$ fixed by $\iota$ and a crepant resolution of such singularities can be obtained blowing up each of these curves. More precisely the following diagram commutes:
\begin{eqnarray}\label{diag: BV construction}
\begin{array}{ccccc}
\iota \circlearrowright & S \times E & \stackrel{\beta}{\leftarrow} & \widetilde{S\times E} & \circlearrowleft\widetilde{\iota}\\
 & \downarrow & & \downarrow\\
 & (S \times E) / \iota & \leftarrow & \widetilde{(S \times E)} / \widetilde{\iota} &\simeq X
\end{array}
\end{eqnarray}
where $\beta: \widetilde{S \times E} \ra S \times E$ is the blow up of $S \times E$ in the fixed locus $\Fix_{\iota}(S \times E)$, $\widetilde{\iota}$ is the involution induced on $S \times E$ by $\iota$ and the vertical arrows are the quotient maps. Thus, a desingularization of $(S \times E) / \iota$ is constructed blowing up the fixed locus $\Fix_{\iota}(S \times E)$ and then considering the quotient by the induced automorphism.
In order to compute the Hodge numbers of the Calabi--Yau 3-fold we first observe that
\[\begin{array}{l}
H^{i}(\widetilde{S \times E}) = H^{i}(S \times E), \qquad i = 0,1,\\
H^{2}(\widetilde{S \times E}) = H^{2}(S \times E) \bigoplus \oplus_{i = 1}^{4N} H^0(D_i),\\
H^{3}(\widetilde{S \times E}) = H^{3}(S \times E) \bigoplus \oplus_{i = 1}^{4N} H^1(D_i), \qquad \text{(cf. \cite{V})},
\end{array}\]
where $D_i$ is the exceptional divisor over the fixed curve $C_i$ blowed up and is isomorphic to a $\PP^1$-bundle over $C_i$.\\
Since $\widetilde{S \times E} / \widetilde{\iota}$ is a smooth quotient of $\widetilde{S \times E}$, the cohomology groups of $X$ coincide with the invariant part of the cohomology groups of $\widetilde{S \times E}$ under $\widetilde{\iota}$. We notice that the exceptional divisors over the fixed locus of $\iota$ are clearly invariant under $\widetilde{\iota}$ in $\widetilde{S\times E}$.

\begin{proposition}[cf.\ \cite{V}]\label{prop: hodge numbers X2}
Let $S$ be a K3 surface admitting a non-symplectic involution $\iota_S$ fixing $N$ curves and let $N' = \sum_{C_i \in \Fix_{\iota_S}(S)} g(C_i)$. Let $E$ be an elliptic curve and $\iota_E$ its hyperelliptic involution. The Hodge numbers of any crepant resolution of $(S \times E) / (\iota_S \times \iota_E)$, and in particular the ones of $X$, are
\[h^{0, 0} = h^{3, 0} = 1, \quad h^{1, 0} = h^{2, 0} = 0, \quad h^{1, 1} = 1 + r + 4 N, \quad h^{2, 1} = m - 1 + 4 N'.\]
Equivalently
\[\begin{array}{ll}
h^{1, 1} = 11 + 5 N - N' = 5 + 3 r - 2 a,\\
h^{2, 1} = 11 + 5 N' - N = 65 - 3 r - 2 a,
\end{array}\]
where $a$ is defined by the following property: $(H^2(S, \Z)^{\iota_S})^{\vee}/(H^2(S, \Z)^{\iota_S}) \simeq (\Z / 2\Z)^a$.
\end{proposition}

The proposition follows immediately by the construction of $X$ of type $X_2$ given before, by \eqref{formula: invariant h11 h21} and by the following known relations among $(r, a)$ and $(N, N')$: $N = \frac{1}{2} (r - a + 2)$, $N' = \frac{1}{2} (22 - r - a)$.

\begin{rem}
If $N' = 0$, $\iota_X$, induced on $X$ by $\iota_S\times \id$, is a maximal automorphism, by Proposition \ref{prop: h21=m-1 maximal automorphism}.  
\end{rem}
\begin{rem} If $N'=0$ any automorphism $\sigma_S$ commuting with $\iota_S$ induces a maximal automorphism of $X$ by Remark \ref{rem: other maximal automorphisms}. In \cite{GS} it is proved that, if  $\iota_S$ is such that $N'=0$, then $S$ admits at least one symplectic involution, $\sigma_S$. So, if $N'=0$, we obtain at least one maximal automorphism (in fact an involution) preserving the period of $X$. We observe that there exists a crepant resolution of the quotient of $X$ by such an automorphism which is a Calabi--Yau 3-fold, and it is again of Borcea--Voisin type. However it is not in general a deformation of $X$.
\end{rem}

\subsection{The elliptic fibration}\label{sec: elliptic fibration 2}
We now consider the map  $\mathcal{E}_2:X\ra S/\iota_S$ (cf. Proposition \ref{prop: fibrations on Xn}) which turns out to be an elliptic fibration on $X$, a Calabi--Yau of type $X_2$, and whose analogue will be considered in the following sections. 

\begin{proposition}\label{prop: ell. fib. 2}
Let $\cE_2: X \ra S/\iota_S$ be the natural map induced on $X$ by $(S \times E) / (\iota_S \times \iota_E) \ra S/\iota_S$.

The map $\cE_2: X \ra S/\iota_S$ is an isotrivial elliptic fibration whose general fiber is isomorphic to $E$. Let us consider the quotient map $q: S \ra S/\iota_S$. Let us assume that $S$ is generic in the family of K3 surfaces with the non-symplectic involution $\iota_S$ (i.e.\ $\rho(S) = r$). The fiber $F_P$ of $\cE_2$ over $P \in S/\iota_S$ is singular if and only if $P$ is in the branch locus of $q: S \ra S/\iota_S$ and in this case $F_P$ is of type $I_0^*$. The Mordell--Weil group of this fibration is generically trivial.
\end{proposition}
\begin{proof}
Since $\iota_S$ does not fix isolated points, the quotient $S / \iota_S$ is smooth. The generic fiber of the map $\cE_2: X \ra S/\iota_S$ is an elliptic curve isomorphic to $E$ by construction. The discriminant locus of $\cE_2$ is the branch locus of $q: S \ra S/\iota_S$ and so is isomorphic to $\Fix_{\iota_S}(S)$. Any of its components is a copy of a curve $C$ fixed by $\iota_S$. By the construction of $X$ we introduce four $\PP^1$-bundles over $C$ for each curve $C$, which are in fact the blow up of $C \times P_i$, $i = 1, 2, 3, 4$. Thus, if we consider the fiber $F_P$ over a point $P$ of $q(C)$ we find the strict transform of $E$, which is a rational curve, and 4 rational curves which are the fibers over the point $P$ of the four $\PP^1$-bundles over $C \simeq q(C)$. Hence we find exactly a fiber of type $I_0^*$.
\end{proof}

\begin{rem}
The automorphism $\alpha_X$ defined in Definition \ref{defi: alphaX} is the hyperelliptic involution on the elliptic fibration $\cE_2: X \ra S/\iota_S$.
\end{rem}

\subsection{Non-symplectic automorphisms of order 2 on K3 surfaces}\label{subsect: covering of f4}
Every pair $(S, \iota_S)$, where $S$ is a K3 surface admitting a non-symplectic involution $\iota_S$, can be described in one of the following ways:
\begin{enumerate}
\item $S$ is the minimal resolution of the double cover of $\PP^2$ branched over a (possibly singular) sextic and $\iota_S$ is induced on $S$ by the cover involution;
\item $S$ is an elliptic fibration with section and $\iota_S$ is induced by the hyperelliptic involution on each smooth fiber;
\item $S$ is the minimal resolution of the double cover of $\PP^1 \times \PP^1$ branched over a (possibly singular) curve of bi-digree $(4,4)$ and $\iota_S$ is induced on $S$ by the cover involution.
\end{enumerate}

In the second case, and more precisely if $S$ is generic among the K3 surfaces admitting an elliptic fibration with section, $S/\iota_S$ is the Hirzebruch surface $\F_4$. Indeed $S$ can be embedded in $\PP(\cO_{\PP^1}(4) \oplus \cO_{\PP^1}(6) \oplus \cO_{\PP^1})$ with a Weierstrass equation $y^2 z = x^3 + A_8(s : t)xz^2 + B_{12}(s : t)z^3$, with $A_8(s : t) \in H^0(\PP^1, \cO_{\PP^1}(8))$ and $B_{12}(s : t) \in H^0(\PP^1, \cO_{\PP^1}(12))$. The quotient by the hyperelliptic involution corresponds to the projection of the K3 surface $S$ from the constant section $(x: y: z) = (0: 1: 0)$ on the surface $y = 0$.
This defines a $2: 1$ covering of $\F_4$ ($y = 0$ is a surface isomorphic to $\F_4$), branched along $(x^3 + A_8(s: t) xz^2 + B_{12}(s: t) z^3) z = 0$. So the branch divisor in $\F_4$ is $12 D_t + 4 D_z$ and is the disjoint union of the curve $D_z$, which corresponds to the section of the elliptic fibration on $S$, and of the curve $12 D_t + 3 D_z$ which is the image of the trisection passing through the points of order 2 of the elliptic fibration on $S$. The first is a rational curve, the latter has genus 10. In particular the surface $S$ admits the following equation $w^2 = (x^3 + A_8(s: t) x z^2 + B_{12}(s: t) z^3) z$, where $(s:t:x:z)$ are the coordinates of $\F_4$ introduced in Section \ref{sec: Hirzebruch} (more precisely, $x$ was $y$ with the notations of the Section \ref{sec: Hirzebruch}).

Hence a (possibly singular) model of $S$ has one of the following equations: 
\begin{enumerate}
\item $w^2 = f_6(x_0:x_1:x_2)$; \item $w^2 = (x^3 + A_8(s: t) x z^2 + B_{12}(s: t) z^3) z$;\item $w^2 = p_{4,4}((x_0: x_1),(y_0: y_1))$,\end{enumerate} and in all these cases $\iota_S$ acts trivially on all the coordinates but $w$ and changing the sign of $w$.

\subsection{Equations}\label{Weierstrass 2}
Let us now assume $E$ has the following equation $v^2 = u^3 + a u + b$ and $S$ is the double cover of $\PP^2$ branched along a sextic $V(f_6(x_0:x_1:x_2))$. An equation for a singular model of $X_2$ is
\begin{equation}\label{eq: Weierstrass order 2 P2}
Y^2 = X^3 + a f_6^2(x_0:x_1:x_2) X + b f_6^3(x_0:x_1:x_2)
\end{equation}
where the functions $Y := v w^3$, $X := u w^2$ are invariant for $\iota_S \times \iota_E: ((w, (x_0: x_1: x_2)), (v, u)) \mapsto ((-w, (x_0: x_1: x_2)), (-v, u))$. 

Similarly, if $S$ is a double cover of ${\F_4}_{(s:t:x:z)}$, a Weierstrass equation for the elliptic fibration $\cE_2 \ra \F_4$ is
\begin{eqnarray}\begin{array}{rl}\label{eq: Weierstrass order 2 F4}
Y^2 = & X^3 + a (x^3 + A_8(s: t) x z^2 + B_{12}(s: t) z^3)^2 z^2 X +\\
 & + b (x^3 + A_8(s: t) x z^2 + B_{12}(s: t) z^3)^3 z^3.
\end{array}\end{eqnarray}

If $S$ is the double cover of $\PP^1 \times \PP^1$ branched along the curve $V(p_{4,4}((x_0: x_1), (y_0: y_1)))$, an equation for a singular model of $X_2$ is
\begin{equation}\label{eq: Weierstrass order 2 P1xP1}
Y^2 = X^3 + a p_{4,4}^2((x_0: x_1), (y_0: y_1)) X + b p_{4,4}^3((x_0: x_1), (y_0: y_1)).
\end{equation}

If $V(f_6(x_0:x_1:x_2))$ and $V(p_{4,4}((x_0: x_1), (y_0: y_1)))$ are smooth, then \eqref{eq: Weierstrass order 2 P2},  \eqref{eq: Weierstrass order 2 F4}, \eqref{eq: Weierstrass order 2 P1xP1} are Weierstrass equations for the elliptic fibration $\cE_2: X_2 \ra \PP^2$, $\cE_2:X_2\ra \F_4$ $\cE_2: X_2 \ra \PP^1 \times \PP^1$ described in Proposition \ref{prop: ell. fib. 2} respectively and the results of such a proposition can be directly checked on the equation.

\section{Quotient of order 3}

\subsection{The construction of Calabi--Yau 3-folds of type $X_3$}\label{sec: construction X3}
Let $S$ be a K3 surface admitting a non-symplectic automorphism $\alpha_S$ of order 3. Such K3 surfaces are classified in \cite{AS1}, \cite{T}. The fixed locus of $\alpha_S$ consists of $n$ isolated points and of $k$ curves and the linearization of $\alpha_S$ near the fixed locus is $\diag(\zeta_3^2, \zeta_3^2)$ and $\diag(\zeta_3, 1)$ respectively.\\ 
Let $E_{\zeta_3}$ and its automorphism of order 3, $\alpha_E$, be as in Section \ref{subsec: automorphisms elliptic curve}. The automorphism $\alpha_E$ fixes 3 points of $E_{\zeta_3}$ and its local action near the fixed locus is the multiplication by $\zeta_3$.
Let $\alpha$ be the automorphism $\alpha_S \times \alpha_E^2$ of $S \times E$. The fixed locus of $\alpha$ on $S \times E$ consists of $3n$ points and $3k$ curves, and the linearization of $\alpha$ near the fixed locus is $\diag(\zeta_3^2, \zeta_3^2, \zeta_3^2)$ and $\diag(\zeta_3, 1, \zeta_3^2)$ respectively.
As in the case of the involutions one can construct a desingularization of $(S \times E) / \alpha$ by blowing up $S \times E$ in the fixed locus of $\alpha$ to obtain a variety $\widetilde{S \times E}$ such that the induced automorphism $\widetilde{\alpha}$ fixes only divisors, and then considering the quotient $\widetilde{(S \times E)} / \widetilde{\alpha}$, but in this case one has to contract some divisors on $\widetilde{(S \times E)} / \widetilde{\alpha}$ in order to obtain a Calabi--Yau 3-fold.
One finds that it suffices to blow up once the fixed points, introducing a copy of $\PP^2$ as exceptional divisor. The situation is a little bit more complicated in the case of fixed curves: one has to blow up the fixed curve $C$, introducing a divisor $D_1$, which is a $\PP^1$-bundle over $C$, and the induced automorphism fixes two disjoint sections (copies of $C$) on $D_1$. Hence, one has to blow up these 2 copies of $C$ introducing other two exceptional divisors $D_2$, $D_3$, which are again $\PP^1$-bundles over $C$.
The induced automorphism fixes only divisors and then the quotient by it is smooth. The image of the divisor $D_1$ under the quotient has to be contracted to obtain a smooth Calabi--Yau 3-fold. Globally, we introduce $3n$ exceptional divisors isomorphic to $\PP^2$ over the $3n$ fixed points, and $6k$ exceptional divisors which are $\PP^1$-bundles over $C_i$, two for each fixed curve $C_i$. A more detailed construction of the crepant resolution of $(S \times E) / \alpha$ of type $X_3$ can be found in \cite{CH} and \cite{R}.\\
Examples of Calabi--Yau 3-folds constructed in this way are studied in \cite{R} and \cite{GvG}.

\begin{proposition}\label{prop: Hodge numbers X3}
Let $S$ be a K3 surface admitting a non-symplectic automorphism $\alpha_S$ of order 3, fixing $n$ isolated points and $k$ curves. Let us denote by $C$ the curve of highest genus fixed by $\alpha_S$ and by $g(C)$ its genus. Let $E_{\zeta_3}$ be the elliptic curve with Weierstrass form $v^2 = u^3 + 1$ and $\alpha_E: (v, u) \ra (v, \zeta_3 u)$. Then any crepant resolution of $(S \times E) / (\alpha_S \times \alpha_E^2)$ is a Calabi--Yau 3-fold with Hodge numbers
\[\begin{array}{ll}
h^{0, 0} = h^{3, 0} = 1, & h^{1, 0} = h^{2, 0} = 0,\\
h^{1, 1} = r + 1 + 3 n + 6 k, & h^{2, 1} = m - 1 + 6 g(C).
\end{array}\]
Equivalently
\[h^{1, 1} = 7 + 4 r - 3 a \quad \text{and} \quad h^{2, 1} = 43 - 2 r - 3 a\] where $a$ is defined by the following property: $(H^2(S, \Z)^{\alpha_S})^{\vee} / (H^2(S, \Z)^{\alpha_S}) \simeq (\Z / 3\Z)^a$.\\
In particular the Euler characteristic is $2(h^{1, 1} - h^{2, 1}) = -72 + 12 r$.
\end{proposition}
\begin{proof}
The proposition follows immediately by the construction of the Calabi--Yau of type $X_3$, by the fact that $r + 2 m = 22 = \dim(H^2(S, \C))$ and by the following relations among the pair $(r, a)$ and the invariant $(g(C), k, n)$ describing the fixed locus $\Fix_{\alpha_S}(S)$ of $\alpha_S$ on $S$: $g(C) = (22 - r - 2 a) / 4$, $k = (6 + r - 2 a) / 4$, $n = r/2 - 1$ (\cite[Section 2]{AS1}).
\end{proof}

Let $X$ be a Calabi--Yau of type $X_3$.
\begin{rem}
If $g(C) = 0$, then $h^{2, 1}(X) = m - 1$. So, by Proposition \ref{prop: h21=m-1 maximal automorphism} the automorphism $\alpha_X$ is a maximal automorphism of the family of Calabi--Yau 3-folds $X$, and this family does not admit maximal unipotent monodromy. For this reason, the families of Calabi--Yau 3-folds constructed in this way are analyzed both in \cite{R} and \cite{GvG}.
\end{rem}

\begin{rem}\label{rem: variationa Hodge structure 3}
More precisely, in \cite{GvG} it is proved that every K3 surface $S$ such that $g(C) = 0$ is birational to $(D \times E_{\zeta_3}) / (\Z/3\Z)$ where $D$ is an appropriate curve. Hence the Calabi--Yau 3-fold $X$ is birational to $(D \times E_{\zeta_3} \times E_{\zeta_3}) / (\Z / 3\Z)^2$. Using that the curve $E_{\zeta_3}$ is rigid, one proves that the variation of the Hodge structures  of $X$ depends only on the variation of the Hodge structures of the curve $D$, see \cite{GvG}. We will see in the following that this result generalizes to other Calabi--Yau's of type $X_n$. 
\end{rem}

\begin{rem}
If $g(C) = 0$ and $k\geq 3$, then $S$ admits also a symplectic automorphism $\sigma$ of order 3 which commutes with $\alpha_S$, as proved in \cite{GS}. So we are in the assumption of Remark \ref{rem: other maximal automorphisms} and thus both $\alpha_S\times \id\in \Aut(S \times E)$ and  $\sigma \times \id \in \Aut(S \times E)$ induce maximal automorphisms of the family of $X$; the first one does not preserve the period, the second one, denoted by $\sigma_X$, preserves the period. In particular there exists a crepant resolution, $Y$, of $X/\sigma_X$ which is again a Calabi--Yau 3-fold. Since $S/\sigma$ is birational to a K3 surface admitting a non-symplectic automorphism of order 3, $Y$ is birational to a Calabi--Yau 3-fold of type $X_3$.
\end{rem}

\subsection{(Almost) elliptic fibrations}
Let $\alpha_S$ be a non-symplectic automorphism of $S$ of order 3 which fixes $n$ isolated points and $k$ curves. Let $X$ be the Calabi--Yau 3-fold of type $X_3$ associated to $S$. Let us consider the map $\cE_3: X \ra S/\alpha_S$ induced on $X$ by $(S \times E)/(\alpha_S \times \alpha_E^2) \ra S/\alpha_S$. Moreover we consider the quotient map $q_S: S \ra S/\alpha_S$ and we recall that $S/\alpha_S$ has exactly $n$ singular points, the image of the $n$ isolated fixed points on $S$. We will assume $S$ to be generic in the family of K3 surfaces with a non-symplectic automorphism of order 3 (i.e.\ $\rho(S)=r$).

\begin{proposition}\label{prop: ell. fib. 3}
The map $\cE_3: X \ra S/\alpha_S$ is an almost elliptic fibration whose general fiber is isomorphic to $E_{\zeta_3}$ and it is an elliptic fibration if and only if $n = 0$. The fiber $F_P$ of $\cE_3$ over $P \in S/\alpha_S$ is of dimension 1 if and only if $P$ is a smooth point and is singular if and only if $P$ is in the branch locus of $q: S \ra S/\alpha_S$. If $P$ is a singular point of $S/\alpha_S$, the fiber $F_P$ consists of a rational curve which meets three disjoint copies of $\PP^2$, each in 1 point (see  Figure \ref{figure: non kodaira fibre - order 3}). If $P$ is a smooth point of $S/\alpha_S$ in the brach locus of $q$, then $F_P$ is a fiber of type $IV^*$. The Mordell--Weil group of this fibration is generically equal to $\Z/3\Z$.
\end{proposition}
\begin{proof}
Again the proof follows directly by the construction of $X$. Let us denote by $R_i$ the 3 fixed points of $\alpha_E$ on $E$. We denote by $P' \in S$ the point such that $q(P') = P$. If $P' \in S$ is an isolated fixed point of $\alpha_S$, the points $P' \times R_i$, are isolated fixed points of $\alpha_S \times \alpha_E^2$. On each of them we introduce a $\PP^2$. Moreover, the image of the strict transform of $P' \times E_{\zeta_3}$ is a rational curve which is a component of the fiber $F_P$ over $P$.
\begin{center}
\begin{figure}[h]
\includegraphics[width = 0.2\textwidth]{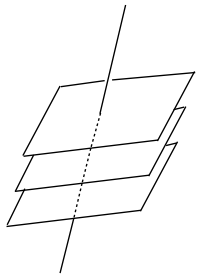}
\caption{The fiber over a singular point of $S/\alpha_S$.}
\label{figure: non kodaira fibre - order 3}
\end{figure}
\end{center}
Similarly, one constructs the reducible fiber over $P$ if $P'$ lies on a curve $C'$ fixed by $\alpha_S$.
In order to construct $X$, we introduce six $\PP^1$-bundles over $C'$, 2 on each curve $C' \times R_j$, $j = 1, 2, 3$. Considering the fiber of these $\PP^1$-bundles over $P' \times R_j$ we obtain 3 copies of $A_2$ (where $A_2$ is a configuration of 2 rational curves meeting in a point). Moreover, there is a component of the fiber $F_P$, which consists of the image of the strict transform of $P' \times E_{\zeta_3}$. It meets in one point one of the two rational curves of each configuration of type $A_2$. Hence we obtain  a configuration of $(-2)$-curves which corresponds to a fiber of Kodaira type $IV^*$.\\
The (rational) sections are the image of the strict transform of the divisors $S \times R_j$. It is immediate to show that they are sections if $S/\alpha_S$ is smooth and rational sections otherwise.
\end{proof}

\begin{rem}
The automorphism $\alpha_X$ defined in Definition \ref{defi: alphaX} is induced by the complex multiplication of order 3 on each smooth fiber of the fibration.
\end{rem}

We now give a Weierstrass equation of $\mathcal{E}_3$ in case $S/\alpha_S$ is smooth. This implies that $\alpha_S$ has no isolated fixed points, i.e. $n=0$.  By Proposition \ref{prop: ell. fib. 3}, this condition is equivalent to the condition $\mathcal{E}_3$ is an elliptic fibration. 

By \cite{AS1}, \cite{T} there exist exactly two families of K3 surfaces admitting a non-symplectic automorphism of order 3 which do not fix isolated points. Each of these families is 9 dimensional. They are characterized by the number of fixed curves of $\alpha_S$ on $S$. If $\alpha_S$ fixes exactly one curve on $S$, then $S/\alpha_S\simeq \mathbb{P}^1\times \mathbb{P}^1$ and if $\alpha_S$ fixes exactly 2 curves on $S$, then $S/\alpha_S\simeq \mathbb{F}_6$.

\subsubsection{Weierstrass equation of $\mathcal{E}_3$ if  $S/\alpha_S\simeq \PP^1\times\PP^1$}\label{Weierstrass 3-P1P1}
Let $S$ be a K3 surface admitting a non-symplectic automorphism of order 3 whose fixed locus consists of exactly one curve. In this case the fixed curve has genus 4 and an equation for $S$ is (\cite[Proposition 4.7]{AS1})
\begin{eqnarray}\label{eq: K3 in P4}
\left\{\begin{array}{lll}
F_2(x_0:x_1:x_2:x_3) & = & 0\\
F_3(x_0:x_1:x_2:x_3)+x_4^3 & = & 0
\end{array}\right.
\end{eqnarray}
where the polynomials $F_n(x_0: x_1: x_2: x_3)$ are generic polynomials of degree $n$. In this case $\alpha_S$ is induced by the projectivity $(x_0: x_1: x_2: x_3: x_4) \ra (x_0: x_1: x_2: x_3: \zeta_3 x_4)$. It is now clear that $S$ is a $3:1$ cover of a quadric, $V(F_2(x_0: x_1: x_2: x_3)) \subset \PP^3$, branched along the curve $V(F_2(x_0: x_1: x_2: x_3)) \cap V(F_3(x_0: x_1: x_2: x_3)) \subset \PP^3$. More intrinsically, $S$ is the triple cover of $\PP^1 \times \PP^1$ branched along a curve of bidigree $(3, 3)$, i.e.\ $S$ admits an equation of type $w^3 = p_{3, 3}((x_0: x_1), (y_0: y_1))$ and we can assume $\alpha_S \colon (w, ((x_0: x_1), (y_0: y_1))) \mapsto (\zeta_3 w, ((x_0: x_1), (y_0: y_1)))$.

So a Weierstrass equation for the elliptic fibration described in Proposition \ref{prop: ell. fib. 3} is
\[Y^2 = X^3 + p_{3, 3}^4((x_0: x_1), (y_0: y_1)),\]
where the functions $Y := v w^6$, $X := u w^4$ are invariant for $\alpha_S \times \alpha_E^2$. We can directly check that this elliptic fibration has three sections, two of them have order 3 and are $(X, Y)\mapsto (0, \pm p_{3, 3}^2((x_0: x_1),(y_0: y_1)))$, and that the reducible fibers are of type $IV^*$.

\subsubsection{Weierstrass equation of $\mathcal{E}_3$ if  $S/\alpha_S\simeq \mathbb{F}_6$}\label{Weierstrass 3 F6}
Let $S$ be a K3 surface admitting a non-symplectic automorphism of order 3 whose fixed locus consists exactly of 2 curves. In \cite[Proposition 4.2]{AS1} it is proved that $S$ admits an isotrivial elliptic fibration whose equation is $y^2 = x^3 + p_{12}(t)$ where $p_{12}(t)$ is a polynomial of degree 12 without multiple roots. The automorphism $\alpha_S$ in this case is $(x, y, t) \mapsto (\zeta_3 x, y, t)$. The fixed curves are the section of this elliptic fibration (which is of course rational) and the bisection $y^2 = p_{12}(t)$.
Using the homogeneous coordinates in $\PP(\cO_{\PP^1}(4) \oplus \cO_{\PP^1}(6) \oplus \cO_{\PP^1})$, the surface $S$ has equation $y^2 z = x^3 + p_{12}(s: t) z^3$ and so we see that the projection from the constant section $(x: y: z) = (1: 0: 0)$ defines a $3:1$ cover of the surface $x = 0$, which is the Hirzebruch surface $\F_6$. Using the coordinates $(s: t: y: z)$ on $\F_6$ introduced in Section \ref{sec: Hirzebruch}, the branch locus is given by the two disjoint curves $z = 0$, i.e.\ the negative curve of $\F_6$, which corresponds to the section of the elliptic fibration, and $y^2 = p_{12}(s: t) z^2$ which corresponds to the bisection.\\
Observe that $S$ admits a birational model $S'$ in $\cO_{\F_6}(4 D_t + D_z)$, whose equation is $w^3 = (y^2 - p_{12}(s: t) z^2) z$. The birational morphism
\[\begin{array}{ccc}
S & \longrightarrow & S'\\
(s, t, x, y, z) & \longmapsto & (w, (s: t: y: z)) = (x, (s: t: y: z))
\end{array}\]
is compatible with the non-symplectic automorphism $\alpha_S$, in the following sense: via this birational morphism, $\alpha_S$ induces on $S'$ the covering transformation $\alpha_{S'}: (w, (s: t: y: z)) \longmapsto (\zeta_3 w, (s: t: y: z))$ which is non-symplectic.\\
The functions $Y := v w^6$ and $X := u w^4$ are then invariant for $\alpha_{S'} \times \alpha_E^2$ and satisfy
\[Y^2 = X^3 + (y^2 z - p_{12}(s:t)z^3)^4,\]
which is a Weierstrass equation for the elliptic fibration $\mathcal{E}_3: X_3 \ra S / \alpha_S\simeq \mathbb{F}_6$ in $\PP(\cO_{\F_6}(-2 K_{\F_6}) \oplus \cO_{\F_6}(-3 K_{\F_6}) \oplus \cO_{\F_6})$.

\section{Quotient of order 4}
\subsection{The construction of Calabi--Yau 3-folds of type $X_4$}\label{sec: construction X4}

Let $E_i$ be the elliptic curve admitting an automorphism, $\alpha_E$, of order 4 such that $\alpha_E^*(\omega_E) = i\omega_E$. We recall that $\alpha_E^2$ is the hyperelliptic involution and thus fixes 4 points. Among these points, two are switched by $\alpha_E$ and two are fixed. We denote by $P_i$, $i = 1, 2$ the points such that $\alpha_E(P_i) = P_{i}$ and by $Q_j$, $j = 1, 2$ the points such that $\alpha_E^2(Q_j) = Q_j$ and $\alpha_E(Q_1) = Q_2$.

Let $S$ be a K3 surface admitting a purely non-symplectic automorphism of order 4, $\alpha_S$. Such K3 surfaces are not completely classified, but a lot of them are studied and listed in \cite{AS2}. The fixed locus of $\alpha_S$ consists of points and curves. Of course $\Fix_{\alpha_S}(S) \subset \Fix_{\alpha_S^2}(S)$. Since $\alpha_S^2$ is a non-symplectic involution of $S$ and thus fixes only curves, all the points fixed by $\alpha_S$ lie on curves fixed by $\alpha_S^2$.

We will denote by $D$ the curve of highest genus fixed by $\alpha_S^2$ and by $g(D)$ it genus. The number of curves fixed by $\alpha_S^2$ is $N$. We denote by $n_1$ the number of isolated fixed points of $\alpha_S$ which are not points of the curve $D$ and by $n_2$ the number of isolated fixed points of $\alpha_S$ which are on $D$.

The automorphism $\alpha_S$ can act in three different ways on each curve fixed by $\alpha_S^2$. We call:
\begin{enumerate}
\item curves of first type: the curves which are fixed by $\alpha_S$ (and thus of course also by $\alpha_S^2$), we denote by $A_i$ these curves and by $k$ their number;
\item curves of second type: the curves which are fixed by $\alpha_S^2$ and are invariant by $\alpha_S$, we denote by $B_i$ these curves and by $b$ their number;
\item curves of third type: the curves fixed by $\alpha_S^2$ and sent to another curve by $\alpha_S$, we denote by $C_i$ these curves and by $2a$ their number. In particular we assume $\alpha_S(C_i) = C_{i + a}$, $i = 1,\ldots, a$.
\end{enumerate}

Observe that the isolated fixed points of $\alpha_S$ lie on curves of second type and that $\alpha_S$ is an involution of each of these curves. Clearly $N = k + b + 2a$.

Let $\alpha := \alpha_S \times \alpha_E^3\in \Aut(S \times E_i)$. In order to construct a crepant resolution of the quotient $(S \times E_i) / \alpha$ we first consider the quotient $(S \times E_i) / \alpha^2$: it is of the type described in Section \ref{sec: construction X2}, and so one can construct a crepant resolution $X'$ of $(S\times E_i)/\alpha^2$ by blowing up once the fixed locus and then considering the quotient by the induced automorphism. We denote by $\alpha'$ the automorphism induced by $\alpha$ on $X'$. It has order 2 and preserves the period of $X'$ because $\alpha$ preserves the period of $S \times E_i$. The local action of $\alpha'$ near the fixed locus can be linearized to $\diag(-1, 1, -1)$.
Thus, in order to construct a crepant resolution of $X'/ \alpha'$ it suffices to blow up $X'$ in the fixed locus of $\alpha'$ and then to consider the quotient by the automorphism induced by $\alpha'$ on the blow up. We denote by $X$ the crepant resolution obtained in this way and we will say that it is of type $X_4$. We have the following diagram:
\[\xymatrix{S \times E_i \ar[d] & \widetilde{S \times E_i} \ar[l]_{\beta_1} \ar[d]_{\pi_1} & \\
S \times E_i / \alpha^2 \ar[d] & \widetilde{S \times E_i} / \tilde{\alpha^2} = X' \ar[l] \ar[d] & \widetilde{X'} \ar[l]_(.33){\beta_2} \ar[d]_{\pi_2}\\
S \times E_i / \alpha & X' / \alpha' \ar[l] & \widetilde{X'} / \tilde{\alpha'} = X\ar[l]}\]
where $\beta_1$ is the blow up of $S \times E_i$ in $\Fix_{\alpha^2}(S \times E)$, $\beta_2$ is the blow up of $X'$ in $\Fix_{\alpha'}(X')$ and all the vertical arrows are the quotient maps and are $2:1$.\\
The computation of the cohomology groups of $X$ is similar to the one described in the previous examples, but one has to be a little bit more careful with the fixed loci of $\alpha^2$ and $\alpha$.

Let us denote by $G_i$ the curves fixed by $\alpha_S^2$. The automorphism $\alpha^2$ fixes the $4N$ curves $G_i \times P_j$, $G_i \times Q_j$, $i = 1, \ldots, N$, $j = 1, 2$. The blow up $\beta_1$ introduces a $\PP^1$-bundle on each of them. All these divisors are preserved (by construction) by the quotient $\widetilde{S \times E_i} / \alpha^2$.

The automorphism $\alpha$ acts non-trivially on the set of $2N$ curves $G_i \times Q_j$, $i = 1,\ldots, N$, $j = 1, 2$.
So the fixed locus of $\alpha'$ on the $\PP^1$-bundle over these curves is empty. Moreover, since $\alpha'$ has order 2, the quotient by $\alpha'$ identifies pairs of such divisors.

Now we have to consider the action of $\alpha$ on the curves $G_i \times P_j$, $i = 1, \ldots, N$, $j = 1, 2$. It depends on the action of $\alpha_S$ on $G_i$.

If $G$ is a curve of the first type, $\alpha$ fixes $G \times P_j$. So the base of the $\PP^1$-bundle introduced on $G \times P_j$ is fixed. Since the fibers of such a bundle are rational curves and $\alpha'$ is an involution on each of these curves, there are exactly two sections of this $\PP^1$-bundle which are fixed by $\alpha'$. Each of these sections is isomorphic to $G$. So we have to blow up two curves isomorphic to $G$, and thus $\beta_2$ introduces other 2 exceptional divisors, which are $\PP^1$-bundles on $G$. By construction these two exceptional divisors are invariant for $\alpha'$ and thus they correspond to two divisors on $X$, which are $\PP^1$-bundles over $G$.

If $G$ is a curve of the second type, then for each isolated fixed point $V \in G$ of $\alpha_S$, $V\times P_j$ is a fixed point of $\alpha$. So there is a fiber (the one over $V$) of the $\PP^1$-bundle over $G$ which is fixed by $\alpha'$. Hence, for every isolated fixed point of $\alpha_S$, the blow up $\beta_2$ introduces a $\PP^1$-bundle over $\PP^1$ (the latter is a fiber of a $\PP^1$-bundle over $G$). The quotient by $\alpha'$ modifies also the exceptional divisors of $\beta_1$. Indeed, since $\alpha$ acts non-trivially on $G$, $\alpha'$ acts non-trivially on the basis of the exceptional $\PP^1$-bundle over $G$ introduced by $\beta_1$. The image of such a divisor under the quotient by $\alpha'$ is a $\PP^1$-bundle over $G/\alpha_S$.

If $G$ is a curve of the third type, then there exists another curve $G'$ such that $\alpha_S(G) = G'$. In particular $\alpha(G \times P_j) = G' \times P_j$, $j = 1, 2$. So there are no fixed points for $\alpha'$ on the $\PP^1$-bundle over $G$, and the quotient by $\alpha'$ identifies the $\PP^1$-bundle over $G$ with the one over $G'$.

Thanks to this construction we can explicitly compute the Hodge numbers of $X$. In order to consider the contribution to the Hodge numbers given by the exceptional divisors we apply \cite[Th\'eor\`em 7.3.1]{V2}, see also \cite[Proposition 7]{G}, obtaining
\begin{align}\label{formula: Hodge 4 example}
\begin{array}{rcl}
H^j(X) & = & H^j(S \times E_i), \qquad j = 0, 1,\\
H^j(X) & = & H^j(S \times E_i)^\alpha \oplus\\ &&\oplus_{i = 1}^k \left(  \oplus_{l = 1}^3 \left(H^{j - 2}(\coprod{D_{A_i \times P_1}} \coprod D_{A_i \times P_2})\right)\oplus H^{j - 2}(D_{A_i \times Q_1}) \right) \oplus\\
 & & \oplus_{i = 1}^b \left( H^{j - 2}(\coprod D_{(B_i \times P_1)/\alpha} \coprod D_{(B_i \times P_2)/\alpha} \coprod D_{B_i \times Q_1}) \right) \oplus \\&&  \oplus_{i = 1}^{a} (H^{j - 2}(\coprod D_{C_i \times P_1} \coprod D_{C_i \times P_2} \coprod D_{C_i \times Q_1} \coprod D_{C_i \times Q_2})),\oplus\\
&&\oplus_{i = 1}^{n_1 + n_2} \left( \oplus_{l = i}^2 H^{j - 2}(D_{\PP^1}) \right)  \qquad j = 2, 3\\
\end{array}
\end{align}
where $D_C$ is a $\PP^1$-bundle over the curve $C$ and we recall that $C \times Q_1 \simeq C \times Q_2$ and $C_i \simeq C_{i + a}$.

\begin{proposition}\label{prop: Borcea Voisin 4}
Let $S$ be a K3 surface with a purely non-symplectic automorphism $\alpha_S$, such that $\alpha_S^*(\omega_S) = i\omega_S$. Let us denote by $m := \dim(H^2(S, \Z)_i) = \dim(H^2(S, \Z)_{-i})$ and by $r := \dim(H^2(S, \Z)^{\alpha_S})$. \\
Let $X$ be a crepant resolution of $(S\times E_{i})/(\alpha_S\times\alpha_E^3).$

If $\alpha_S^2$ fixes two elliptic curves, then the Hodge numbers of $X$ are $h^{j, 0} = 1$, if $j = 0, 3$,  $h^{j, 0} = 0$, if $j = 1, 2$,  $h^{1, 1} = 19 + r = 25$, $h^{2, 1} = m - 1 + 8 = 13$.

If the fixed locus of $\alpha_S^2$ does not consist of two elliptic curves, let us denote by $D$ the curve of highest genus fixed by $\alpha_S^2$. There are the following two possibilities:
1) $D$ is of the first type; 2) $D$ is of the second type.

The Hodge numbers of $X$ are:
\[\begin{array}{lll}
h^{j, 0} = 1 \text{ if } j = 0, 3, & h^{j, 0} = 0  \text{ if } j = 1, 2, & \\
h^{1, 1} = 1 + r + 7k + 3b+2 (n_1+n_2) + 4a, & h^{2, 1} = m - 1 + 7g(D), & \text{in case 1)}\\
h^{1, 1} = 1 + r + 7k + 3b+2 (n_1+n_2)+ 4a, & h^{2, 1} = m + 2g(D) - n_2 / 2, & \text{in case 2)}.
\end{array}\]
Equivalently,
\[\begin{array}{lll}
h^{1, 1} = 22 + 17 k + 5 a - 10 g(D), &
h^{2, 1} = 4 - k - a + 8 g(D), & \text{in case 1)}\\
h^{1, 1} = \frac{102 - 7 n_2 - 2 g(D)}{4} + 17 k + 5 a, &
h^{2, 1} = \frac{18 - n_2 + 10 g(D)}{4} - k - a, & \text{in case 2).}
\end{array}\]
In particular $\chi(X) = 36 + 36 k - 36 g(D) + 12 a$ in case 1), $\chi(X) = 42 - 3 n_2 - 6 g(D) + 36 k + 12 a$ in case 2) and $\chi(X) = 24$ if $\alpha_S^2$ fixes 2 elliptic curves.
\end{proposition}
\begin{proof}
The proof of the first part of the statement follows directly by the construction of $X$ and by \eqref{formula: Hodge 4 example}. One has to observe that the genus of $D/\alpha$ is $g(D)$ if $D$ is of the first type. If $D$ is of the second type, the genus of $D/\alpha$ can be computed by Riemann--Hurwitz formula since $D$ is a double cover of $D/\alpha$ branched in $n_2$ points. Thus $g(D/\alpha) = (2 g(D) + 2 - n_2)/4$.

In order to obtain the Hodge numbers as function of $(n_1, n_2, g(D), k, a)$ one needs the following relations among these numbers and $(r, m)$:
\[r = \frac{1}{2} (12 + k + 2a + b - g(D) + 4 h), \quad m = \frac{1}{2} (12 - k - 2a - b + g(D))\]
where $h = \sum_{C \subseteq \Fix \alpha_S(S)} (1 - g(C))$, which are computed by \cite[Thm. 1, p. 4]{AS2} and \cite[Prop. 1, p. 4]{AS2}. Moreover we observe that if $D$ is of the first type $h = (k - 1) + (1 - g(D)) = k - g(D)$, $n_2 = 0$, $n_1 = 2h + 4$, $b = n_1 / 2$; if $D$ is of the second type, $h = k$, $n_1 + n_2 = 2h + 4$, $b = n_1 / 2 + 1$.  
\end{proof}

\begin{rem}\label{rem: no mum 4}
If $g(D)=0$, then $h^{2, 1}(X) = m - 1$. So, by Proposition \ref{prop: h21=m-1 maximal automorphism},  if $D$ is rational $\alpha_X$ is a maximal automorphism of the family of Calabi--Yau 3-folds of $X$, and this family does not admit maximal unipotent monodromy. Some of these families were already constructed in \cite{G} (the ones corresponding to lines 12 - 13 - 15 - 16 - 18 of Table \ref{table: Hodge numbers case 1}).

We observe moreover that any other automorphism on $S$ commuting with $\alpha_S$ induces a maximal automorphism of $X$. In particular if $X$ admits an isotrivial elliptic fibration with fibers isomorphic to $E_i$ (this surely happens for some $S$, for example in many cases constructed in \cite{G}), then $S$ admits a symplectic involution (the translation by the 2-torsion section of this elliptic fibration) and thus $X$ admits a maximal automorphism of order 2 preserving the period.
\end{rem} 

\begin{rem}\label{rem: isotrivial fib. S4, no mum} In \cite{G} it is proved that some K3 surfaces admitting a non-symplectic automorphism of order 4 such that $D$ is a rational curve (with the notation of Proposition \ref{prop: Borcea Voisin 4}), are birational to the quotient $(C \times E_i) / (\Z/4\Z)$ where $C$ is a certain curve of positive genus. As a consequence one obtains that $X$ is birational to $(C \times E_i \times E_i)/ (\Z/4\Z)$ and thus the variation of the Hodge structures of $X$ depends only on the variation of the Hodge structures of the curve $C$, see also Remark \ref{rem: variationa Hodge structure 3}. Here we can extend this result to each K3 surface admitting a non-symplectic automorphism $\alpha_S$ of order 4 such that $D$ is a rational curve and $\alpha_S$ is the cover automorphism of a 4:1 map $S\ra\mathbb{P}^2$. Indeed, if $S$ is a 4:1 cover of $\mathbb{P}^2$ the branch locus consists of a  plane quartic curve $Q$, which is fixed by $\alpha_S$. Since $D$ is rational, $Q$ is quite 
singular (it could also be reducible) and in particular it 
admits at least either one node or a cusp, say in the point $P\in Q$. The pencil of lines through $P$ in $\mathbb{P}^2$ induces a pencil of curves on $S$. The general member of this pencil is a $4:1$ cover of $\mathbb{P}^1$ branched in 2 points of order 4 and two points of order 2, thus it is isomorphic to $E_i$. So, the pencil of lines through $P$ in $\mathbb{P}^2$ induces an isotrivial elliptic fibration on $S$ whose general fiber is isomorphic to $E_i$. Hence there exists a curve $C$ such that $S$ is birational to $(C \times E_i) / (\Z/4\Z)$, $X$ to $(C \times E_i \times E_i)/ (\Z/4\Z)$ and the variation of the Hodge structures of $X$ depends only on the one of the curve $C$.\end{rem}
\begin{rem}
Let $S$ be a K3 surface, which is a 4:1 cover of $\mathbb{P}^2$ and let $\alpha_S$ be the cover automorphism. The branch locus of $S\xrightarrow{4:1}\mathbb{P}^2$ is a possibly singular quartic curve $Q$. Let us denote by $x$ the number of nodes of $Q$, by $y$ the number of cusps of $Q$ and by $z$ the number of the components of $Q$. Then the properties of the fixed locus of $\alpha_S$ are the following: $D$ is a curve of first type and is the component of $Q$ of highest genus, $k=z$, $N=k+x+y$, $n=2(x+y)$, $a=y$. 
Hence the Calabi--Yau 3-folds of type $X_4$ constructed by $S$ are the ones described in lines 2,3,4,5,7,8,11,\ldots,19 of Table \ref{table: Hodge numbers case 1}. In particular the Calabi--Yau 3-folds which are of the type described in Remark \ref{rem: isotrivial fib. S4, no mum} are the ones of lines $12,\ldots,19$ of Table \ref{table: Hodge numbers case 1}.\end{rem}

\begin{rem} There exists at least three different Calabi--Yau 3-folds of Borcea--Voisin type which are rigid and admit an automorphism acting as $\zeta_4=i$ on $\omega_X$ (see lines 18, 19 of Table \ref{table: Hodge numbers case 1} and line 5 of Table \ref{table: only rational curves}). Since these Calabi--Yau 3-folds are rigid, $H^3(X,\Q)\simeq \Q^2$ and the Hodge structure is given by the decomposition in eigenspaces for the action of $\alpha_X$ (i.e. for the action of $\cdot i$ on $\Q^2\simeq \Q(i)$). These Calabi--Yau 3-folds are associated to the same K3 surface $S$ (the unique K3 surface admitting a non--symplectic involution, $\alpha_S^2$, whose fixed locus consists of 10 rational curves). On this $S$ there are at least 3 automorphisms  $\alpha_S\in \Aut(S)$ of order 4 such that $\alpha_S^2$ fixes 10 rational curves and for each of them the Borcea--Voisin construction gives a different Calabi--Yau 3-fold.
For example, assume that $S$ admits both a $4:1$ cover of $\mathbb{P}^2$ branched along 4 lines and a  $4:1$ cover of $\mathbb{P}^2$ branched along a quartic with three cusps. The cover automorphisms in these two cases are different and have different fixed loci (these are associated to the Calabi--Yau 3-folds in lines 18 and 19 in Table \ref{table: Hodge numbers case 1}).\end{rem}

\subsection{(Almost) elliptic fibrations}

Let $X$ be a Calabi--Yau 3-fold of type $X_4$. Let us consider the map $\cE_4: X \ra S/\alpha_S$ induced on $X$ by $(S \times E_i)/(\alpha_S \times \alpha_E^3) \ra S/\alpha_S$. We consider the quotient map $q_S: S \ra S/\alpha_S$ and we recall that $S/\alpha_S$ has exactly $n_1 + n_2$ singular points, the images of the $n_1 + n_2$ isolated fixed points on $S$.
We will assume $S$ to be generic in the family of K3 surfaces with a non-symplectic automorphism of order 4 (i.e.\ $\rho(S) = 22 - 2m$).

\begin{proposition}\label{prop: ell. fib. 4}
The map $\cE_4: X \ra S/\alpha_S$ is an almost elliptic fibration whose general fiber is isomorphic to $E_i$ and it is an elliptic fibration if and only if $n_1 + n_2 = 0$. The fiber $F_R$ of $\cE_4$ over $R \in S/\alpha_S$ is of dimension 1 if and only if $R$ is a smooth point, and is singular if and only if $R$ is in the branch locus of $q: S \ra S/\alpha_S$. In particular:
\begin{itemize}
\item If $q^{-1}(R)$ consists of two points, the fiber $F_R$ is of type $I_0^*$;
\item If $q^{-1}(R)$ consists of one point and $R \in S/\alpha_S$ is smooth, the fiber $F_R$ is of type $III^*$;
\item If $R$ is a singular point of $S/\alpha_S$, the fiber $F_R$ consists of 2 rational curves meeting in a point and two disjoint divisors which are $\PP^1$-bundles over $\PP^1$, and which intersect both the same rational curve in one point (see Figure \ref{figure: non kodaira fibre - order 4}).
\end{itemize}
The Mordell--Weil group of this fibration is generically equal to $\Z / 2\Z$.
\end{proposition}
\begin{proof}
With the same notations as in the previous section, we denote by $\cE': X' \ra S/\alpha_S^2$ the map induced on $X'$ by $(S \times E_i) / (\alpha_S^2 \times \alpha_E^2) \ra S/\alpha_S^2$ and by $F'_R$ the fiber of $\cE'$ over $R \in S/\alpha_S^2$. Moreover we denote by  $q': S \ra S/\alpha_S^2$ the quotient map.

Let $R \in S/\alpha_S$ be such that $q^{-1}(R) = \{R_1, R_2\} \subset S$. This implies $q'(R_i) = R_i$, $i = 1, 2$ and so the fibers $F'_{R_i}$ are of type $I_0^*$, by Proposition \ref{prop: ell. fib. 2}. The automorphism induced by $\alpha_S$ on $S/\alpha_S^2$ switches $R_1$ and $R_2$. So $\alpha'$ switches $F_{R_1}$ and $F_{R_2}$ and thus $F_R$ is a fiber of type $I_0^*$.

Let $R \in S/\alpha_S$ be a smooth point such that $q^{-1}(R) = R_1 \in S$. This implies that there exists a curve $A \subset S$ such that $R_1 \in A$ and $A \subset \Fix_{\alpha_S}(S)$. The fiber $F'_{R}$ of $\cE'$ is of type $I_0^*$ (by Proposition \ref{prop: ell. fib. 2}). Two of the components with multiplicity one (those which are fibers of the $\PP^1$-bundles introduced by $\beta_1$ over $A \times Q_1$ and $A \times Q_2$ respectively) are switched by $\alpha'$. The automorphism $\alpha'$ acts as an involution on the other 2 components of multiplicity 1 and fixes two points on each of them. So the fiber $F_R$ of $\cE_4$ is a fiber of type $III^*$ and 7 of its 8 components are fibers of the seven $\PP^1$-bundles over $A$ introduced by the blow ups $\beta_1$ and $\beta_2$. The other component is the image of the strict transform of $E_i$.

Let $R \in S/\alpha_S$ be a singular point. Then $q^{-1}(R) = R_1$, $R_1$ is an isolated fixed point for $\alpha_S$ and it lies on a curve $B$ fixed by $\alpha_S^2$. The fiber $F'_R$ of $\cE'$ is of type $I_0^*$ (by Proposition \ref{prop: ell. fib. 2}). Two of the components of multiplicity 1 of $I_0^*$ are switched by $\alpha'$ (the ones which are fibers of the $\PP^1$-bundles introduced by $\beta_1$ over $B \times Q_1$ and $B \times Q_2$ respectively). The other two are fibers of the $\PP^1$-bundles introduced by $\beta_1$ over $B \times P_1$ and $B \times P_2$. These curves are fixed by $\alpha'$ and so $\beta_2$ is a blow up of each of these curves. Hence the fiber $F_R$ consists of a rational curve which is the image of the strict transform of $E_i$, of a rational curve which is the image of the two rational curves switched by $\alpha'$ and of 2 divisors, which are exceptional divisors of $\beta_2$ and are $\PP^1$-bundles over $\PP^1$.
In particular it contains 2 curves and 2 divisors (see Figure \ref{figure: non kodaira fibre - order 4}).

\begin{center}
\begin{figure}[h]
\includegraphics[width = 0.4\textwidth]{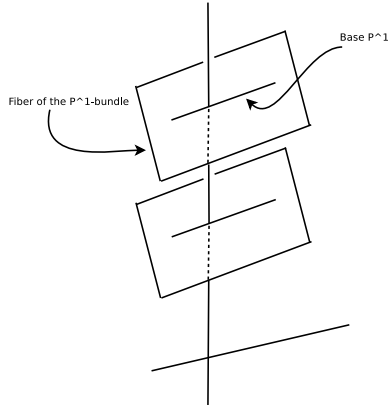}
\caption{The fiber over a singular point of $S/\alpha_S$.}
\label{figure: non kodaira fibre - order 4}
\end{figure}
\end{center}

The (rational) sections are the images of $S \times P_1$ and $S \times P_2$ under the rational map $\pi_2 \circ \beta_2^{-1} \circ \pi_1 \circ \beta_1^{-1}$. One of them can be chosen as zero section. The other is a section of order 2, indeed if $P_1$ is the zero of the elliptic curve $E_i$, then $P_2$ has order 2 on $E_i$.
\end{proof}

We now give the Weierstrass equations of $\mathcal{E}_4$ in case the base $S/\alpha_S$ is smooth. This is equivalent to require that $\alpha_S$ has no isolated fixed points, i.e. $n_1+n_2=0$. We observe that in this case $\mathcal{E}_4$ is an elliptic fibration, by Proposition \ref{prop: ell. fib. 4}. 

By \cite[Proposition 1]{AS2}, $n_1+n_2=0$ implies $h = -2$ where $h = \sum_{C \in \Fix_{\alpha_S}(S)} (1 - g(C))$ ($h$ is the same as $\alpha$ in \cite{AS2}). We already noticed that if $D$ is of the second type $h = k$, i.e.\ $h$ is the number of curves fixed by $\alpha_S$. Since $k \geq 0$, if $n_1 + n_2 = 0$, then $D$ is of first type. In this case $-2 = h = k + 1 - g(D)$, which implies $g(D) \geq 3$. By \cite[Theorem 4.1]{AS2}, $g(D) \leq 3$. So $n_1 + n_2 = 0$ implies that $D$ is of the first type and $g(D) = 3$. There exist exactly two families of K3 surfaces admitting a non-symplectic automorphism of order 4 fixing a curve of genus 3, and the difference among these cases is that in one case $D$ is hyperelliptic while in the other it is not. These families are respectively 5 and 6 dimensional. We now consider both these situation.

\subsubsection{Weierstrass equation of $\mathcal{E}_4$ if  $S/\alpha_S\simeq \PP^2$}\label{Weierstrass 4 P2-proj}

If $D$ is not a hyperelliptic curve, the K3 surface is a $4:1$ cover of $\PP^2$ branched along a smooth quartic. Let us denote by $Q$ the smooth quartic in $\PP^2$, which is the zero locus of $f_4(x_0: x_1: x_2: x_3)$. The equation of $S$ is $t^4 = f_4(x_1: x_2: x_3)$. The automorphism $\alpha_S: S \ra S$ is the cover automorphism $(t: x_1: x_2: x_3) \mapsto (it: x_1: x_2: x_3)$.

The variables  $Y := v t^9$, $X := u t^6$, defined on $S \times E_i$, are invariant for $\alpha_S \times \alpha_E^3$ and satisfy
\begin{equation}\label{formula: elliptic on X}
Y^2 = X^3 + X f_4^3(x_1: x_2: x_3).
\end{equation}

Moreover, the map $((v, u), (t, x_0: x_1: x_2)) \ra ((x, y), (x_1: x_2: x_3))$ is $4:1$ on $S \times E_i$ and thus \eqref{formula: elliptic on X} is the Weierstrass equation of the elliptic fibration $\mathcal{E}_4: X \ra S / \alpha_S$ described in Proposition \ref{prop: ell. fib. 4}. Since the discriminant is $\Delta = f_4^3(x_1:x_2:x_3)$ all the singular fibers are of type $III^*$. This in fact agrees with Proposition \ref{prop: ell. fib. 4}.

\subsubsection{Weierstrass equation of $\mathcal{E}_4$ if  $S/\alpha_S\simeq \mathbb{F}_4$}\label{Weierstrass 4 F4}
We now assume that $\alpha_S$ fixes a hyperelliptic curve of genus 3 on $S$. We first briefly recall the construction of $S$ as $4:1$ cover of $\mathbb{F}_4$ given in \cite[Section 2]{A}. Then we deduce an equation for the elliptic fibration $\mathcal{E}_4$ in this case.

Let $C$ be a hyperelliptic curve of genus $3$. Then $C$ is defined by an equation of the form $y^2 = f_8(s: t)$ with $f_8 \in H^0(\PP^1, \cO_{\PP^1}(8))$ in the line bundle $\cO_{\PP^1}(4)$, where $(s: t)$ are coordinates on $\PP^1$ while $y$ is a coordinate on the fiber.
We embed $\cO_{\PP^1}(4)$ in $\PP^5$ as the cone on the rational normal curve in $\PP^4 = \{ x_5 = 0 \}$ with vertex $(0: 0: 0: 0: 0: 1)$. Blowing up this cone in the singular point gives us the Hirzebruch surface $\Hirz_4$. The toric resolution of $\mathcal{O}_{\mathbb{P}^1}(4)$ (with coordinates $(s:t:y)$) is $\mathbb{F}_4$ with coordinates $(s:t:y:z)$ as in Section \ref{sec: Hirzebruch}. With this coordinates, an equation for $C$ is
\[C: y^2 = f_8(s: t) z^2.\]
Let $\widetilde{\Hirz_4}$ be the double cover of $\Hirz_4$ branched along $C$, and $S$ the double cover of $\widetilde{\Hirz_4}$ branched along the strict transforms of $C$ and $D_z$. Then $S$ is a K3 surface, which is a $4:1$ covering of $\Hirz_4$. Moreover the covering automorphism fixes the strict transform of $C$ and swithces the strict transforms of $D_z$.\\
Introducing a new coordinate $\eta$, an equation for $\widetilde{\F_4}$ is $\eta^2 = y^2 - f_8(s: t) z^2$ and so an equation for $S$ is
\[\left\{ \begin{array}{l}
\xi^2 = \eta z\\
\eta^2 = y^2 - f_8(s: t) z^2
\end{array} \right.\]
The covering automorphism is $(\eta, \xi, (s: t: y: z)) \longmapsto (-\eta, i \xi, (s: t: y: z))$.\\
The surface $S$ admits a model $S'$ in $\cO_{\F_4}(2 D_t + D_z)$, with equation $w^4 = (y^2 - f_8(s:t) z^2) z^2$, and the birational morphism
\[\begin{array}{ccc}
S & \ra & S'\\
(\eta, \xi, (s: t: y: z)) & \longmapsto & (w, (s: t: y: z)) = (\xi, (s: t: y: z))
\end{array}\]
is compatible with $\alpha_S$ and the covering transformation $\alpha_{S'}: (w, (s: t: y: z)) \longmapsto (iw, (s: t: y: z))$.\\
On the product $S' \times E$ the functions $Y := \frac{u w^9}{z^3}$ and $X := \frac{u w^6}{z^2}$ are invariant under the action of the morphism $\alpha$. Moreover they satisfy the relation
\[Y^2 = X^3 + (y^2 - f_8(s: t) z^2)^3 z^2 X,\]
which is an equation for the Weierstrass model of $\mathcal{E}_4: X \ra S / \alpha_S$ in $\PP(\cO_{\F_4}(-2 K_{\F_4}) \oplus \cO_{\F_4}(-3 K_{\F_4}) \oplus \cO_{\F_4})$.

\section{Quotient of order 6}\label{sec: X6}
\subsection{The construction of Calabi--Yau 3-folds of type $X_6$}\label{subsec: X6}
Let us consider the elliptic curve $E_{\zeta_3}$ introduced in Section \ref{sect: cy and automorphism}. We already observed that its Weierstrass equation is $v^2 w = u^3 + w^3$ and that it has an automorphism $\beta_E: (u: v: w) \ra (\zeta_3^2 u: -v: w)$, whose square is $\beta_E^2 = \alpha_E$.\\
We consider a K3 surface $S$ which admits a purely non-symplectic automorphism $\beta_S$ of order 6. As shown in Proposition 
\ref{prop: the construction} there exists a desingularization of $(S \times E_{\zeta_3})/(\beta_S \times \beta_E^5)$ which is a Calabi--Yau. Observe that, since $S$ admits both a non-symplectic automorphism of order 3 ($\beta_S^2$) and one of order 2 ($\beta_S^3$), one can construct from $S$ special members of the families of Calabi--Yau 3-folds of type $X_2$ and of type $X_3$ described in Sections \ref{sec: construction X2} and \ref{sec: construction X3} respectively. In \cite{R3}, the author analyzes the consequence of this fact on Calabi--Yau 3-folds constructed from some special K3 surfaces $S$.

Let us denote by $\beta := \beta_S \times \beta_E^5$.
In order to construct a crepant resolution of $(S \times E_{\zeta_3})/\beta$ one can use the previous results in two different ways:
\begin{enumerate}
\item one considers the automorphism $\beta^2 \in \Aut(S \times E_{\zeta_3})$ of order 3 and constructs a crepant resolution, $\widetilde{S \times E_{\zeta_3}/\beta^2}$, of $(S \times E_{\zeta_3})/\beta^2$ as in Section \ref{sec: construction X3}. Then, one considers the automorphism $\beta'$ of order 2  induced on  $\widetilde{S \times E_{\zeta_3}/\beta^2}$ by $\beta$ and constructs a crepant resolution of $\left( \widetilde{S \times E_{\zeta_3}/\beta^2} \right) / \beta'$ as in Section \ref{sec: construction X2};
\item\label{item: x6 construction} one considers the automorphism $\beta^3 \in \Aut(S \times E_{\zeta_3})$ of order 2 and constructs a crepant resolution, $\widetilde{S \times E_{\zeta_3}/\beta^3}$, of $(S \times E_{\zeta_3})/\beta^3$ as in Section \ref{sec: construction X2}. Then, one considers the automorphism $\beta'$ of order 3  induced on  $\widetilde{S \times E_{\zeta_3}/\beta^3}$ by $\beta$ on and constructs a crepant resolution of $\left( \widetilde{S \times E_{\zeta_3}/\beta^3} \right) / \beta'$ as in Section \ref{sec: construction X3}.
\end{enumerate}

These two constructions give of course birational Calabi--Yau 3-folds (both are birational to the singular 3-fold $(S \times E_{\zeta_3})/\beta$), which are not necessarily isomorphic.

We will denote by $X$ the Calabi--Yau threefold constructed as in (\ref{item: x6 construction}) and call it of type $X_6$. So we have the following diagram:

\[\xymatrix{
S \times E_{\zeta_3} \ar[d] & \widetilde{S \times E_{\zeta_3}} \ar[l]_{b_1} \ar[d]_{\pi_1} & &\\
S \times E_{\zeta_3} / \beta^3 \ar[d] & \widetilde{S \times E_{\zeta_3}} / \tilde{\beta^3} = X' \ar[l] \ar[d] & \widetilde{X'} \ar[l]_(.3){b_2} \ar[d]_{\pi_2}&\\
S \times E_{\zeta_3} / \beta & X' / \beta' \ar[l]_{a_1} & \widetilde{X'} / \tilde{\beta'} \ar[l]_{a_2} \ar[r]^{\gamma}&X_6 }\]
where $b_1$ is the blow up of $S \times E_{\zeta_3}$ in $\Fix_{\beta^3}(S \times E_{\zeta_3})$, $b_2$ is the blow up of $X'$ in $\Fix_{\beta'}(X')$, the maps $\pi_1$ and $\pi_2$ are $2:1$ and $3:1$ respectively and $\gamma$ is a contraction. Indeed, in order to construct a crepant resolution of the singularities of type $A_2$, we introduce three exceptional divisors, we make the quotient and then we contract the image of one of these divisors (cf. Section \ref{sec: construction X3}). The maps $a_1$ and $a_2$ are defined by the commutativity of the diagram.

In order to compute the cohomology of $X$ we consider the previous diagram and study the divisors introduced by the blow ups $b_1$ and $b_2$, as we did in Section \ref{sec: construction X4} for Calabi--Yau of type $X_4$.

We first briefly describe the properties of a purely non-symplectic automorphism $\beta_S$ of order 6 on a K3 surface $S$ and in particular of its fixed9 locus $\Fix_{\beta_S}(S)$. The study of such automorphisms is the topic of \cite{D}. The fixed locus of $\beta_S$ consists of $l$ curves and $p_{(2, 5)}+p_{(3, 4)}$ points (\cite{D}). The points $p_{(i, j)}$ are the ones such that the linearization of $\beta_S$ near the fixed point is represented by the diagonal matrix $\diag(\zeta_6^i,\zeta_6^j)$.
All the points fixed by $\beta_S$ lie on curves fixed by $\beta_S^3$. Indeed $\beta_S^3$ is a purely non-symplectic automorphism of order 2, and thus its fixed locus consists of $N$ disjoint curves and no isolated points.
The automorphism $\beta_S^2$ is a purely non-symplectic automorphism of order 3. It fixes $k$ curves and $n$ isolated points. In particular there are $p_{(2, 5)}$ points which are fixed both by $\beta_S$ and by $\beta_S^2$, $2n' = n - p_{(2, 5)}$ points which are fixed by $\beta_S^2$ and switched by $\beta_S$. 
There are $p_{(3, 4)}$ points which are isolated fixed points for $\beta_S$ and lie on curves fixed by $\beta_S^2$: in particular these points are the intersection points between curves fixed by $\beta_S^2$ and curves fixed by $\beta_S^3$.
There are $2a$ curves which are fixed by $\beta_S^2$ and switched by $\beta_S$ and $3b$ curves which are fixed by $\beta_S^3$ and permuted by $\beta_S$.

We will be interested in the curves of positive genus with a non-trivial stabilizer (since they are the unique curves which contribute to $h^{2, 1}(X)$). If there exists a curve with positive genus fixed by $\beta_S$, it is unique and we denote it by $D$; if there exists a curve with positive genus fixed by $\beta_S^2$ it is unique and it is denoted by $B$, if there exist curves of positive genus fixed by $\beta_S^3$, then these are at most two and we denote them by $F_i$, $i = 1, 2$. Clearly these curves can coincide, if there exists a curve $D$ of positive genus fixed by $\beta_S$, then $D = B = F_1$. Moreover we observe that if $B$ or $F_i$, $i = 1, 2$ are not fixed by $\beta_S$, they are invariant by $\beta_S$.

Let us now consider the automorphism $\beta_E$ of $E$. It fixes one point which is the zero of the elliptic curve and will be denoted by $O$. The automorphism $\alpha_E = \beta_E^2$ fixes three points: $O$, $R_1$ and $R_2$ and we observe that $\beta_E(R_1) = R_2$. The automorphism $\beta_E^3$ fixes four points: $O$, $P_1$, $P_2$, $P_3$ and we observe that the points $P_i$ form an orbit under $\beta_E$.

Let us construct explicitly the Calabi--Yau 3-fold $X$ of type $X_6$. We will denote by $\delta: (S \times E_{\zeta_3})/(\beta_S \times \beta_E^5) \dashrightarrow X$ the rational map which is the composition $\gamma \circ (a_2)^{-1} \circ (a_1)^{-1}$ and by $\pi$ the quotient map $S \times E_{\zeta_3} \ra (S \times E_{\zeta_3})/(\beta_S\times \beta_E^5)$.
First we consider a curve $C$ which is fixed by $\beta_S$. Then $b_1$ is a blow up of $C \times O$ which introduces a $\PP^1$-bundle over $C$. Two sections of this bundle are fixed by $\beta'$ (which has order 3), so $\gamma \circ b_2$ introduces 2 other $\PP^1$-bundles for each of this sections. Hence $\delta$ introduces 5 exceptional divisors on $X$ over $\pi(C \times O)$, and all of them are $\PP^1$-bundles over $C$.\\
We recall that the points $R_1$ and $R_2$ are identified by the quotient. Hence $\delta$ introduces 2 exceptional divisors on $X$ over $\pi(C \times R_1)$ ($= \pi(C \times R_2)$), both isomorphic to a $\PP^1$-bundle over $C$.
Moreover, $\delta$ introduces one exceptional divisor in $X$ over $\pi(C \times P_1)$, isomorphic to a $\PP^1$-bundle over $C$. So for every curve $C$ fixed by $\beta_S$ we introduce 8 divisors, isomorphic to $\PP^1$-bundles over $C$.

Let us now assume that $C$ is a curve fixed by $\beta_S^3$ and invariant for $\beta_S$. Let us denote by ${p_{(i, j)}}_C$ the number of fixed points of $\beta_S$ on $C$ of type $p_{(i, j)}$. The blow up $b_1$ introduces one divisor over $C \times O$. For every point ${p_{(2, 5)}}_C$ we find a fiber of the exceptional divisor in the fixed locus of $\beta'$ and so $\gamma \circ b_2$ introduces 2 divisors for every point of type $p_{(2, 5)}$ in $C$.
For every point of type $p_{(3, 4)}$, the fiber of the $\PP^1$-bundle over such a point is invariant for $\beta_S$, but not fixed. It contains two fixed points: one of them is contained in a curve $C'$ fixed by $\beta_S^2$, the other is an isolated fixed point. The point which lie on $C'$ is blown up together with the curve $C'$ (and so this blow up will be considered in the following). The second point has to be blown up, and this introduces a $\PP^2$. So $\delta$ introduces $1 + 2{p_{(2, 5)}}_C + {p_{(3, 4)}}_C$ divisors over $\pi(C \times O)$: one of them is a $\PP^1$-bundle over $C/\beta_S$, $2{p_{2,5}}_C$ are $\PP^1$-bundles over $\PP^1$ and the last ones are ${p_{(3, 4)}}_C$ copies of $\PP^2$.\\
Since $\beta_S^3$ identifies $C \times R_1$ with $C \times R_2$, $b_1$ does not introduce exceptional divisors on the image of such curves. The automorphism $\beta'$ fixes only points on $\pi_1(C \times R_1)$ ($= \pi_1(C \times R_2)$) and these are the images of the points of type $p_{(2, 5)}$. So $\delta$ introduces ${p_{(2, 5)}}_C$ copies of $\PP^2$ over $\pi(C \times R_1)$ ($= \pi(C \times R_2)$).
Moreover, $\delta$ introduces a $\PP^1$-bundle over $C/\beta_S$ over $\pi(C \times P_1)$ ($= \pi(C \times P_i)$, $i = 2, 3$).

Let us now assume that $C$ is a curve fixed by $\beta_S^2$ and invariant for $\beta_S$. All the points on $C$ fixed by $\beta_S$ are of type $p_{(3, 4)}$. The rational map $\delta$ introduces: 2 exceptional divisors over $\pi(C \times O)$, both isomorphic to a $\PP^1$-bundle over $C/\beta_S$; 2 exceptional divisors over $\pi(C \times R_1)$ ($= \pi(C \times O)$), both isomorphic to a $\PP^1$-bundle over $C/\beta_S$ and no divisor over $\pi(C \times P_1)$ ($= \pi(C \times P_i)$, $i = 2, 3$).

Let $(C_1,C_2,C_3)$ be a triple of curves, permuted by $\beta_S$ and such that every curve $C_i$ is fixed by $\beta_S^3$. Then $\delta$ introduces one exceptional divisor over $\pi(C_1 \times O)$ ($= \pi(C_i \times O)$, $i = 2, 3$); one exceptional divisor over $\pi(C_1 \times P_j)$ ($= \pi(C_i \times P_j)$, $i = 2, 3$) for every $j = 1, 2, 3$ and no divisor over $\pi(C_1 \times R_1)$ ($= \pi(C_i \times R_j)$, $i = 1, 2, 3$, $j = 1, 2$). All the divisors introduced are $\PP^1$-bundles over $C_1$.

Let $(C_1,C_2)$ be a pair of curves, permuted by $\beta_S$ and such that every curve $C_i$ is fixed by $\beta_S^2$. Then $\delta$ introduces 2 exceptional divisors over $\pi(C_1 \times O)$ ($= \pi(C_2 \times O)$); 2 exceptional divisors over $\pi(C_1 \times R_j)$ ($= \pi(C_2 \times R_h)$, $\{j,h\}=\{1,2\}$) for every $j = 1, 2$ and no divisor over $\pi(C_1 \times P_1)$ ($= \pi(C_i \times P_j)$, $i = 1, 2$, $j = 1, 2, 3$). All the divisors introduced are $\PP^1$-bundles over $C_1$.

Let $(Q_1,Q_2)$ be a pair of points of $S$ switched by $\beta_S$ and fixed by $\beta_S^2$, then $\delta$ introduces a copy of $\PP^2$ over $\pi(Q_1 \times O)$($= \pi(Q_2 \times O)$), over $\pi(Q_1 \times R_1)$ ($= \pi(Q_2 \times R_1)$) and over$\pi(Q_1 \times R_2)$ ($= \pi(Q_2 \times R_2)$).

\begin{proposition}\label{prop: hodge numbers 6}
With the notations above, the Hodge numbers of $X$ are the following:
\[\begin{array}{ll}
h^{j, 0} = & 1 \text{ if } j = 0, 3, h^{j, 0} = 0 \text{ if } j = 1, 2,\\
h^{1, 1} = & r + 1 + 2l + 2N - 2b + 4k - 2a + 3n' + 3p_{(2,5)} + p_{(3,4)}\\
h^{2, 1} = & m - 1 + 2g(D) + 2(g(B/\beta_S) + g(B)) + \\
           & +g(F_1/\beta_S) + g(F_1) + g(F_2/\beta_S) + g(F_2)
\end{array}\]
where $m := \dim(H^2(S, \Z)_{-\zeta_3}) = \dim(H^2(S, \Z)_{-\zeta^2_3})$ and $r := \dim(H^2(S, \Z)^{\beta_S})$.
\end{proposition}
\begin{proof}
The proof is based on the construction of $X$ described above. For example, we noticed that for every curve in $S$ fixed by $\beta_S$, the rational map $\delta$ introduces 8 divisors, and each of them is a $\PP^1$-bundle on such a curve. Since there are $l$ curves on $S$ fixed by $\beta_S$, we have $8l$ divisors introduced by $\delta$. Thus the number of divisors introduced by $\delta$ is $8l + 2 (N - l - 3b) + 3p_{(2, 5)} + p_{(3, 4)} + 4 (k - l - 2a) + 4b + 6a + 3n' = 2l + 2N - 2b + 4k - 2a + 3n' + 3p_{(2, 5)} + p_{(3, 4)}$, so $h^{1,1}$ is as in the statement.
\end{proof}
\begin{rem}\label{rem: maximal autom 6}
We observe that $h^{2, 1}(X) = m - 1$ if and only if both $\beta_S^2$ and $\beta_S^3$ do not fix curves of positive genus. This happens for at least one family of K3 surfaces, the one constructed in Example \ref{example no mum X6}. If $h^{2,1}(X)=m-1$, then the family of $X$ is maximal Borcea--Voisin and does not admit maximal unipotent monodromy (cf. Proposition \ref{prop: h21=m-1 maximal automorphism}).
\end{rem}

\subsection{(Almost) elliptic fibrations}\label{subsect: almost elliptic x6}
We consider the map $\mathcal{E}_6: X \ra S/\beta_S$ and we denote by $F_Z$ the fiber of such a fibration over a point $Z \in S/\beta_S$. We denote by $\mathcal{E}': X' \ra S/\beta_S^3$ and we denote by $F'_Z$ the fiber of such a fibration over a point $Z \in S/\beta_S^3$. Moreover we consider the quotient maps $q_S: S \ra S/\beta_S$ and $q': S \ra S/\beta_S^3$. We observe that $S/\beta_S$ has $p_{(2, 5)} + n'$ singularities of type $A_2$.\\
We will assume $S$ to be generic in the family of K3 surfaces with a non-symplectic automorphism of order 6 (i.e.\ $\rho(S) = 22 - 2m$).

\begin{proposition}\label{prop: ell. fib. 6}
The map $\mathcal{E}_6: X \ra S/\beta_S$ is an almost elliptic fibration whose general fiber is isomorphic to $E_{\zeta_3}$. The fiber $F_Z$ of $\mathcal{E}_6$ over $Z \in S/\beta_S$ has dimension 1 if and only if $Z$ is a smooth point which is not the image of a point of type $p_{(3, 4)}$ and is singular if and only if $Z$ is in the branch locus of $q: S \ra S/\beta_S$. 
In particular:
\begin{itemize}
\item If $q^{-1}(Z)$ consists of 3 points, $F_Z$ is of type $I_0^*$;
\item If $q^{-1}(Z)$ consists of 2 points and $Z \in S/\beta_S$ is smooth, $F_Z$ is of type $IV^*$;
\item If $q^{-1}(Z)$ consists of 2 points and $Z \in S/\beta_S$ is singular, $F_Z$ contains three disjoint copies of $\PP^2$ and a rational curve meeting each $\PP^2$ in a point (see Figure \ref{figure: non kodaira fibre - order 3});
\item If $q^{-1}(Z)$ consists of 1 point and $Z \in S/\beta_S$ is a smooth point which is not the image of a point of type $p_{(3, 4)}$, $F_Z$ is of type $II^*$;
\item If $q^{-1}(Z)$ consists of 1 point and $Z \in S/\beta_S$ is the image of a point of type $p_{(3, 4)}$, $F_Z$ contains seven rationals curves and a divisor $D \simeq \PP^2$ (see Figure \ref{figure: non kodaira fibre - order 6a});
\item If $q^{-1}(Z)$ consists of 1 point and $Z \in S/\beta_S$ is singular, $F_Z$ contains two rational curves and three divisors, one of them is isomorphic to $\mathbb{P}^2$, the other are $\PP^1$-bundles over $\PP^1$ (see Figure \ref{figure: non kodaira fibre - order 6b}).
\end{itemize}

The Mordell--Weil group of this fibration is generically trivial.
\end{proposition}

\begin{center}
\begin{figure}[h]
\subfigure[\label{figure: non kodaira fibre - order 6a}]%
{\includegraphics[width = 0.3\textwidth]{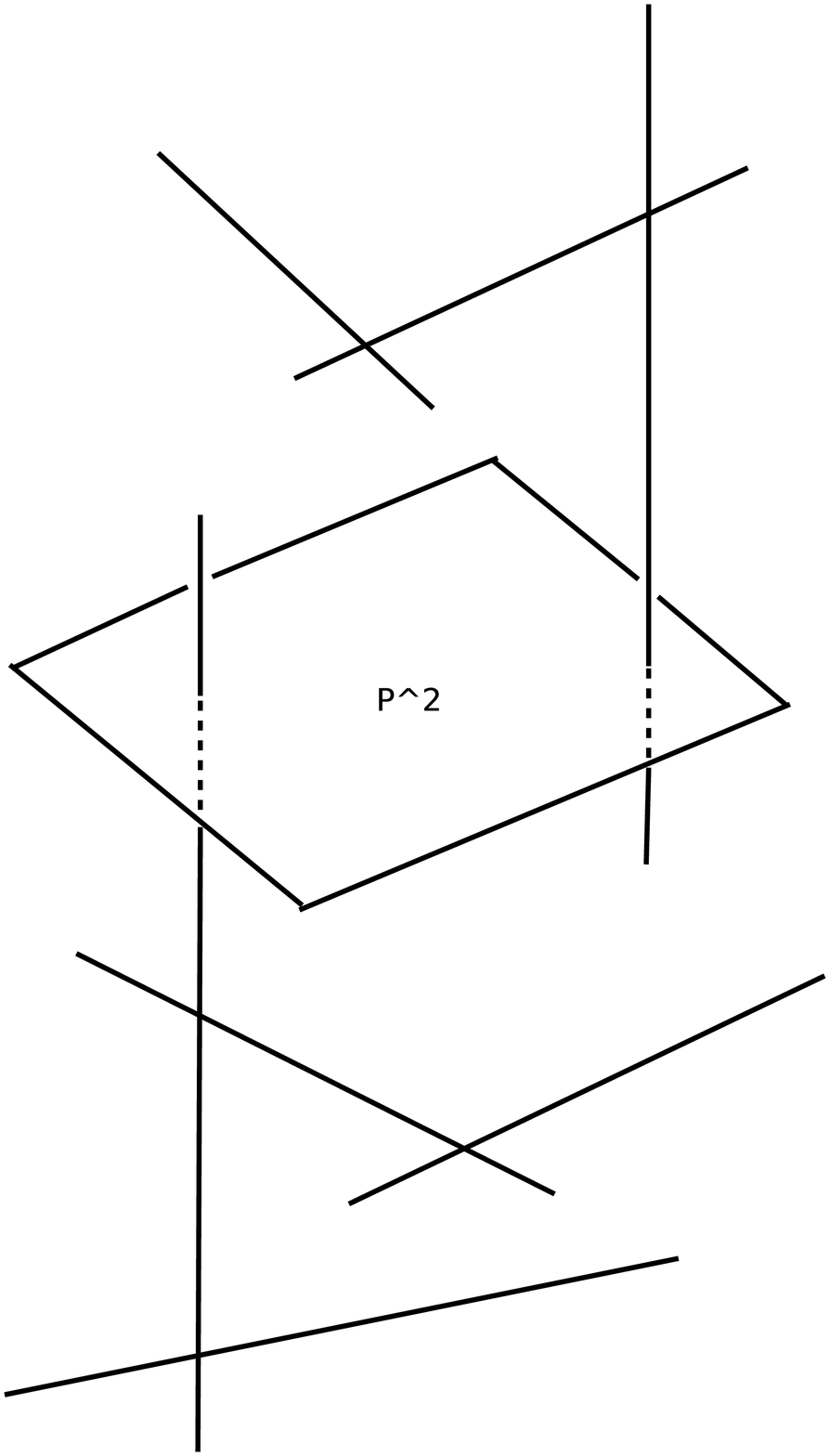}} \qquad
\subfigure[\label{figure: non kodaira fibre - order 6b}]%
{\includegraphics[width = 0.45\textwidth]{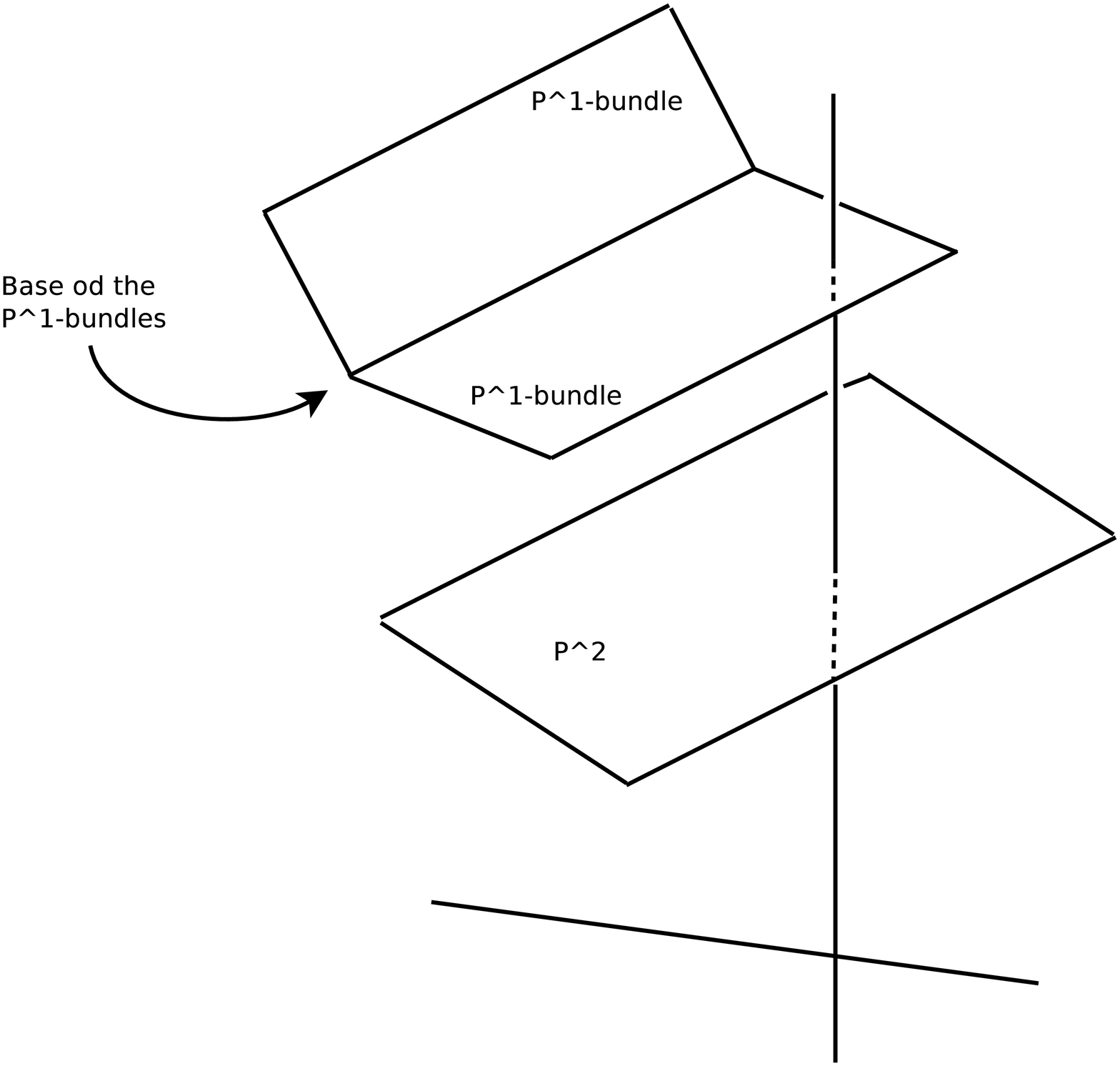}}
\caption{Non-Kodaira fibers of $\mathcal{E}_6$ over the image of a point of type $p_{(3, 4)}$ and $p_{(2, 5)}$ respectively.}
\label{figure: non kodaira fibre - order 6}
\end{figure}
\end{center}

\begin{proof}
%
If $q^{-1}(Z) = \{Z_1, Z_2, Z_3\}$, then the stabilizer of $Z_i \in S$ is the involution $\beta_S^3$. So the fibers of $\mathcal{E}': X' \ra S/\beta^3$ over $q'(Z_i)$, $i = 1, 2, 3$, are 3 fibers of type $I_0^*$ constructed as in Proposition \ref{prop: ell. fib. 2}. They are identified in the quotient by $\beta'$.

If $q^{-1}(Z) = \{Z_1, Z_2\}$, then $\beta_S^3(Z_1) = Z_2$ and $q'(Z_1) = q'(Z_2)$. So the fiber $F'_{q'(Z_1)}$ coincides with the fiber $F'_{q'(Z_2)}$ and they are isomorphic to $E_{\zeta_3}$. Now the automorphism of order 3 $\beta'$ fixes the point $q'(Z_1) = Z$ and the fiber over this point can be constructed as in Proposition \ref{prop: ell. fib. 3}. In particular, if $Z \in S/\beta_S$ is smooth, $F_Z$ is a fiber of type $IV^*$, while if $Z \in S/\beta_S$ is singular, it contains three copies of $\PP^2$.

If $q^{-1}(Z) = \{Z_1\}$, then $\beta_S^3(Z_1) = Z_1$. So the fiber $F'_{q'(Z_1)}$ is a fiber of type $I_0^*$. Let us denote by $E$ the component of multiplicity 2 of this fiber (it is the quotient by $\beta^3$ of the strict transform by $b_1$ of $Z_1 \times E_{\zeta_3} \in S \times E_{\zeta_3}$).
Three of the four components with multiplicity 1 are permuted by $\beta'$ (these are the ones corresponding to the points $Z_1\times P_i \in Z_1 \times E_{\zeta_3}$, $i = 1, 2, 3$), the other is invariant (it is the one corresponding to the point $Z_1\times O \in Z_1 \times E_{\zeta_3}$).

Let us now assume that $Z \in S/\beta_S$ is smooth and that it is not the image of a point of type $p_{(3, 4)}$. On the invariant component $\beta'$ acts as an automorphism of order 3 with 2 fixed points, one of them is the intersection point with the component of multiplicity 2. So $b_2$ blows up these two points. Since $Z$ is smooth, these points lie on curves fixed by $\beta'$, so we introduce 2 rational curves on each point. Moreover there is a point of $E$ which is fixed by $\beta'$ (it is the image of the point $Z_1\times R_1 \in Z_1 \times E_{\zeta_3}$) and again it lies on a curve fixed in $X'$. So, in order to resolve it, we introduce other two rational curves in this fiber.
To recap, the fiber $F_Z$ contains the image of the strict transform of $E$. It intersects: a rational curves which is the image of the three components of $I_0^*$ permuted by $\beta'$; a tree of 5 rational curves which corresponds to the component of $I_0^*$ which is preserved by $\beta'$; a tree of 2 rational curves which corresponds to the point $Z_1\times R_1 \in Z_1 \times E_{\zeta_3}$. The fiber $F_Z$ is then of type $II^*$.

Let $Z \in S/\beta_S$ be the image of a point of type $p_{(3, 4)}$. The fibre of $\mathcal{E}'_6$ over $Z$ is then of type $I_0^*$: the component $E$ of multiplicity 2 is the image of the strict transform of the elliptic curve $E_{\zeta_3}$, and the other four rational curves are fibres of the exceptional divisors ($\PP^1$-bundles) introduced by $b_1$, say $C_1, \ldots, C_4$. 
The automorphism $\beta'$ permutes three of these four rational curves and leaves the other invariant. In particular it fixes 2 points on this component, one of them is the intersection of $C_1$ and $E$ and is an isolated fixed point. The other lies on a curve of fixed points. Hence $b_2$ introduces a copy of $\mathbb{P}^2$ over the isolated fixed point and a tree of two rational curves (fibers of $\mathbb{P}^1$-bundles) over the other fixed point on $C_1$. Moreover there is a point of $E$, which is fixed by $\beta'$ and which lies on a curve of fixed points intersecting transversally $E$. So, $b_2$ introduces a tree of two rational curves (fibers of $\mathbb{P}^1$-bundles) on this point. Thus, we get a configuration of curves as shown in Figure \ref{figure: non kodaira fibre - order 6a}.

Finally, let us now assume that $Z \in S/\beta_S$ is singular. On the invariant component $\beta'$ acts as the identity. So $b_2$ blows up this rational curve introducing divisors which are all contained in the fiber $F_Z$. In particular, on $X$ we have 2 $\PP^1$-bundles over $\PP^1$ meeting along the base. Moreover there is a point of $E$, which is fixed by $\beta'$ (it is the image of the point $Z_1\times R_1\in Z_1 \times E_{\zeta_3}$) and it is an isolated fixed point for $\beta'$ on $X'$. So in order to resolve it we introduce a copy of $\PP^2$.
To recap, the fiber $F_Z$ contains the image of the strict transform of $E$. It intersects: one rational curve which is the image of the three components of $I_0^*$ permuted by $\beta'$; two $\PP^1$-bundles which correspond to the component of $I_0^*$ which is preserved by $\beta'$; one $\PP^2$ which corresponds to the point $Z_1\times R_1\in Z_1 \times E_{\zeta_3}$ (see Figure \ref{figure: non kodaira fibre - order 6b}).

The (rational) section is the images of $S\times O$ under the rational map $\gamma \pi_2 b_2^{-1}\pi_1b_1^{-1}$.
\end{proof}

\begin{rem}
We observe that $\mathcal{E}_6$ would be an elliptic fibration if and only if $n' = p_{(3, 4)} = p_{(2, 5)} = 0$, but this never happens. Indeed by \cite[Theorem 4.1]{D}, $p_{(3, 4)} = p_{(2, 5)} = 0$ implies $6l = -6$, which is absurd.\end{rem}

\begin{rem}
This is the first case where we get 2-dimensional fibers over smooth points of the base surface. The reason why this happens is the following: around a point of type $p_{(3, 4)}$, a local equation for the Weierstrass model of the fibration is
\[y^2 = x^3 + s^3 t^4,\]
and so the origin is a singular point, which is not of $cDV$ type (\cite[Thm. 4.14]{C}). But then to resolve it we must introduce divisors in the corresponding fibre.
\end{rem}

\subsubsection{Weierstrass equation of $\mathcal{E}_6$ if  $S/\alpha_S\simeq \PP^2$}\label{Weeierstrass 6-P2 proj}

The almost elliptic fibration $\mathcal{E}_6$ has a smooth base if and only if $n = 0$ (where $n$ is the number of isolated fixed points of the non-symplectic automorphism of order 3 $\alpha_S = \beta_S^2$). Hence we are interested in particular in this situation. We already observed (Section \ref{Weierstrass 3-P1P1}) that there are two 9-dimensional families of K3 surfaces admitting a non-symplectic automorphism of order 3 without isolated fixed points. We now focus our attention on one of these families (the other will be analyzed in the next section).

Let us assume $\alpha_S = \beta_S^2$ fixes only 1 curve. Then it has genus 4, and $S$ admits a projective model as complete intersection of a quadric and a cubic in $\PP^4$ as in \eqref{eq: K3 in P4}. Since $\beta^2$ has no isolated fixed points, $p_{(2, 5)} = 0$ and thus, by \cite[Theorem 4.1]{D}, we obtain $l = 0$, $p_{(3, 4)} = 6$. So $\beta_S$ fixes exactly 6 points. We now apply the Lefschtez fixed points formula in order to compute the dimension of the family of K3 surfaces admitting a purely non-symplectic automorphism of order 6 whose square fixes only one curve. We observe that $\beta^*$ acts on $H^{2}(S, \C)$ and we denote by $m := \dim(H^2(S, \C))_{\zeta_6^i}$, $i = 1, 5$, $r := \dim(H^2(S, \C))_1$, $a := \dim(H^2(S, \C))_{\zeta_6^j}$, $j = 2, 4$, $b := \dim(H^2(S, \C))_{-1}$. The Lefschetz fixed points formula states that
\[\sum_{i = 0}^4 \operatorname{tr} \left( {\beta_S^*}_{|H^i(S, \C)} \right) = \chi(\Fix_{\beta_S}(S)).\]
We obtain $2 + m - a - b + r = 6$. Moreover we apply the same formula to $\beta_S^2$ and obtain $2 - m - a + b + r = -6$. Since $\dim(H^2(S, \C)) = 22$, we have also $2m + 2a + b + r = 22$. We recall that $m, a, b, r$ are non negative integers (they are the dimension of vector spaces), $m \geq 1$ (because $H^{2, 0}(S) \subset H^{2}(S, \C)_{\zeta_6}$) and $r \geq 1$ because there exists at least one invariant class. So the unique possibilities for $(m, a, b, r)$ are $(7, 3, 1, 1)$ and $(6, 4, 0, 2)$. Since the dimension of the moduli space of the K3 surfaces with the required properties is $m - 1$, we proved that it is at most 6 (in particular not all the K3 surfaces admitting a non-symplectic automorphism $\alpha_S$ of order 3 fixing one curve admit a purely non-symplectic automorphism $\beta_S$ of order 6 such that $\beta_S^2 = \alpha_S$).

Here we show that indeed there exists at least a 6-dimensional family of K3 surfaces with a purely non-symplectic automorphism of order 6 whose square fixes one curve of genus 4, specializing the family given in \eqref{eq: K3 in P4}. Indeed the complete intersections in $\PP^4$ given by the system
\begin{eqnarray}\label{eq: 6:1 to P2}
\left\{ \begin{array}{rrrrr}
F_2(x_0:x_1:x_2) & +x_3^2 & & = 0\\
F_3(x_0:x_1:x_2) & & +x_4^3 & = 0
\end{array}\right.
\end{eqnarray}
are generically smooth and hence K3 surfaces. The projective dimension of such a family is 6 and every member of this family clearly admits an automorphism of order 6, induced by $\beta_{\PP^4}: (x_0: x_1: x_2: x_3: x_4) \ra (x_0: x_1: x_2: -x_3: \zeta_3^2 x_4)$.
We will denote by $S$ the generic K3 surface with equation \eqref{eq: 6:1 to P2} and by $\beta_S$ the automorphism induced by $\beta_{\PP^4}$ on $S$. Observe that $V(F_2(x_0: x_1: x_2) + x_3^2)$ is a quadric in $\PP^3$ and thus it is isomorphic to $\PP^1 \times \PP^1$. Denote by $R_3 := V(F_3(x_0: x_1: x_2) + x_4^3)$. We have the following diagram:

\[\xymatrix{
 & S \ar[dr]_{\pi_2}^{2:1} \ar[dd]_{6:1} \ar[dl]^{\pi_1}_{3:1} & \\
V(F_2(x_0: x_1: x_2) + x_3^2) \simeq \PP^1 \times \PP^1 \ar[dr]_{2:1} & & R_3 := V(F_3(x_0: x_1: x_2) + x_4^3) \ar[dl]^{3:1}\\
 & \PP^2_{(x_0: x_1: x_2)} & }\]
Hence $S/\beta_{S} \simeq \PP^2$. The fixed locus $\Fix_{\beta_S}(S)$ consists of 6 points, as we expect. The fixed locus of $\beta_S^2$ consists of a curve of genus 4, as we required and the fixed locus of $\beta_S^3$ is the $3:1$ cover of the rational curve $V(F_2(x_0: x_1: x_2))\subset \PP^2_{(x_0:x_1:x_2)}$ branched along the six points $V(F_2(x_0: x_1: x_2)) \cap V(F_3(x_0: x_1: x_2))$.

In order to exhibit $S$ as $6:1$ cover of $\PP^2$ we introduce the variable $w:=i x_3 x_4$ and thus an equation of $S$ in $\cO_{\PP^2}(2)$ is
\[w^6 = F_2(x_0: x_1: x_2)^3 F_3(x_0: x_1: x_2)^2\]
and $\beta_S: (w, (x_0: x_1: x_2)) \mapsto (-\zeta_3^2 w, (x_0: x_1: x_2))$.

From the equation of $S$ we deduce a Weierstrass equation for the almost elliptic fibration described in Proposition \ref{prop: ell. fib. 6}:
\[Y^2 = X^3 + F_2(x_0: x_1: x_2)^3 F_3(x_0: x_1: x_2)^4\]
where the functions $Y:= \frac{v w^{15}}{F_2^6 F_3^3}$, $X: = \frac{u w^{10}}{F_2^4 F_3^2}$ are invariant for $\beta_S \times\beta_E^5$.

\subsubsection{Weierstrass equation of $\mathcal{E}_6$ if  $S/\alpha_S\simeq \mathbb{F}_{12}$}\label{Weierstrass 6-F12}
The second family of K3 surfaces admitting a non-symplectic automorphism of order 6 such that $n = 0$, consists of the family of K3 surfaces, $S$, admitting an isotrivial elliptic fibration whose equation is $y^2 = x^3 + p_{12}(t)$ where $p_{12}(t)$ is a generic polynomial of degree 12 without multiple roots. The non-symplectic automorphism we consider is $\beta_S: (x, y, t) \ra (\zeta_3^2 x, -y, t)$. We analyze the generic case (i.e.\ $n = 0$, i.e. $p_{12}(t)$ has no multiple roots) showing that in this case $S$ is a $6:1$ cover of $\mathbb{F}_{12}$ and giving an equation for the almost elliptic fibration $\mathcal{E}_6$ in this case. 

Let $S$ have equation $y^2 z = x^3 + p_{12}(s: t) z^3$ in $\PP(\cO_{\PP^1}(4) \oplus \cO_{\PP^1}(6) \oplus \cO_{\PP^1})$, and consider the morphism to $\F_{12}$ given by $(s: t: x: y: z) \longmapsto (s: t: x^3: z^3)$. Observe that this is the composition of the $2:1$ covering $S \longrightarrow \F_4$ described in Section \ref{subsect: covering of f4} with the $3:1$ covering $\F_4 \longrightarrow \F_{12}$ described in Section \ref{sec: Hirzebruch}, hence it is a $6:1$ covering. Observe also that $S$ admits a birational model $S'$ in $\cO_{\F_{12}}(10 D_T + D_Z)$, with equation $w^6 = X^2 (X + p_{12}(S: T) Z)^3 Z$: the birational morphism
\[\begin{array}{ccc}
S & \longrightarrow & S'\\
(s: t: x: y: z) & \longmapsto & (w, (S: T: X: Z)) = (xyz, (s: t: x^3: z^3))
\end{array}\]
is compatible with $\beta_S$, since it induces on $S'$ the covering automorphism $\beta_{S'}: w \ra -\zeta_3^2 w$.

The functions $\eta := \frac{v w^{15}}{X^3 (X + p_{12}(S: T) Z)^6}$, $\xi := \frac{u w^{10}}{X^2 (X + p_{12}(S: T) Z)^4}$ defined on $S' \times E_{\zeta_3}$ are invariant for $\beta_{S'} \times \beta_E^5$ and satisfy
\[\eta^2 = \xi^3 + X^4 (X + p_{12}(S: T) Z)^3 Z^5\]
which is a Weierstrass equation for $X$ in $\PP(\cO_{\F_{12}}(-2 K_{\F_{12}}) \oplus \cO_{\F_{12}}(-3 K_{\F_{12}}) \oplus \cO_{\F_{12}})$.

\subsection{Invariants of the fixed loci of $\beta_S^j$ for some K3 surfaces}\label{subsec: invariants fixed loci}
In \cite{D} some non-symplectic automorphisms of order 6 on K3 surfaces are classified. In particular, Dillies considers the case where $S$ is a K3 surface with the following equation $y^2 = x^3 + p_{12}(t)$, $p_{12}(t)$ is a polynomial of degree 12 which does not admit roots of multiplicity greeter then 5, and the automorphism of order 6 is $(x, y, t) \mapsto (\zeta_3 x, -y, t)$. In order to compute the Hodge numbers of the Calabi--Yau $X$ of type $X_6$, we want to describe the fixed loci of $\beta_S$, $\beta_S^2$ and $\beta_S^3$ according to the variation of $p_{12}(t)$. Let $\overline{t}$ be a zero of $p_{12}(t)$, $\mu_{\overline{t}}$ the multiplicity of such zero, $f_{\overline{t}}$ the fiber of $y^2 = x^3 + p_{12}(t)$ over $\overline{t}$.
The automorphism $\beta_S$ fixes the zero section; the automorphism $\beta_S^2$ fixes the zero section and the bisection $y^2 = p_{12}(t)$ (in some cases this can split in 2 distinct sections); the automorphism $\beta_S^3$ fixes the zero section and the trisection $x^3 = -p_{12}(t)$ (in some cases this can split either in a section and a bisection or in three sections). All the other fixed curves are components of the reducible fibers.
In Table \ref{table: fixed loci} we give the number of isolated points and of components which are fixed by $\beta_S^j$ on the singular fibers. We do not compute the points which are fixed on the fiber but lie on curves (sections or multi-sections) fixed by the automorphism.
\begin{center}
\begin{longtable}{|c|c|c|c|c|}
\caption{Fixed loci of $\beta_S$ on singular fibers}
\label{table: fixed loci}\\
\hline
$\mu_{\overline{t}}$ & fiber $f_{\overline{t}}$ & $\Fix_{\beta_S}(S)$ & $\Fix_{\beta_S^2}(S)$ & $\Fix_{\beta_S^3}(S)$\\
\hline
1 & $II$ & 1 pt. & - & -\\
2 & $IV$ & 1 pt. & 1 pt. & -\\
3 & $I_0^*$ & 2 pts. & 1 pt. & 1 curve\\
4 & $IV^*$ & 3 pts. & 1 curve, 3 pts. & 1 curve\\
5 & $II^*$ & 1 curve, 7 pts. & 2 curves, 4 pts. & 4 curves\\
\hline
\end{longtable}
\end{center}

Moreover we need to compute explicitly the numbers $r := \dim (H^{2}(S, \C)^{\beta_S})$  and $m := \dim(H^{2}(S, \C))_{-\zeta_3}$. This computation is done by applying the Lefschetz fixed points formula to $\beta_S$, $\beta_S^2$ and $\beta_S^3$ (cf. Section \ref{Weeierstrass 6-P2 proj}).
With all these information, one can compute the invariants related to the automorphism described in \cite{D}. The results of these computations are shown in Table \ref{table: Hodge numbers X6}.

\subsection{A family of Calabi--Yau 3-folds of type $X_6$ without maximal unipotent monodromy}\label{sec: example order 6 no mum}
We now construct another example of a K3 surface admitting a non-symplectic automorphism of order 6 by specializing $y^2 = x^3 + p_{12}(t)$. The peculiarity of this example is that the family of Calabi--Yau 3-folds constructed does not admit maximal unipotent monodromy, since it satisfies the condition of Remark \ref{rem: maximal autom 6}.

\begin{example}\label{example no mum X6}
Let us assume $p_{12}(t) = t^4 (t - 1)^3 (t + 1)^3 (t - \lambda)^2$ with $\lambda \neq 0, \pm 1$. The elliptic K3 surface whose Weierstrass equation is $y^2 = x^3 + t^4 (t - 1)^3 (t + 1)^3 (t - \lambda)^2$ has 4 reducible fibers: $f_0$ is of type $IV^*$, $f_1$ and $f_{-1}$ are  of type $I_0^*$, $f_{\lambda}$ is of type $IV$. Applying the results described in Section \ref{subsec: invariants fixed loci}, one obtains that $\beta_S$ fixes one rational curve (the zero section) and 8 isolated points (3 on $f_0$, 2 on $f_i$ for $i = 1, -1$ and 1 on $f_\lambda$). The automorphism $\beta_S^2$ fixes 3 rational curves (the zero section, the bisection $y^2 = t^4 (t - 1)^3 (t + 1)^3 (t - \lambda)^2$ and one component of $f_0$) and 6 isolated points (3 on $f_0$, 1 on $f_i$ for $i = 1, -1, \lambda$). The automorphism $\beta_S^3$ fixes 5 rational curves (the zero section, the trisection $x^3 = -t^4 (t - 1)^3 (t + 1)^3 (t - \lambda)^2$, 1 component of each fiber $f_i$, $i = 0, 1, -1$).
So the Euler characteristic of the fixed loci of $\beta_S$, $\beta_S^2$ and $\beta_S^3$ are 8, 12, 10 respectively. Applying the Lefschetz fixed points formula one computes $r = 11$ and $m = 2$. So the family of K3 surfaces admitting a non-symplectic automorphism of order 6, such that the fixed loci of all its powers are as described, is 1-dimensional. Since also the family of K3 surfaces admitting an elliptic fibration with equation $y^2 = x^3 + t^4 (t - 1)^3 (t + 1)^3 (t - \lambda)^2$ is 1-dimensional, these two families coincide.\\
We observe that the family of K3 surfaces described is obtained also as the minimal model of the quotient $(E_{\zeta_3} \times C) / (\beta_E \times \beta_C)$, where $C$ is a $6:1$ cover of $\PP^1_t$ with equation $w^6 = t^4 (t - 1)^3 (t + 1)^3 (t - \lambda)^2$ and $\beta_C$ is the cover automorphism $\beta_C: (w, t) \mapsto (-\zeta_3 w, t)$. As in Remark \ref{rem: variationa Hodge structure 3} this implies that the variation of the Hodge structures both of the family of $S$ and of the family of $X$ depends only on the variation of the Hodge structures of $C$.
\end{example}

\section{Appendix: the Tables}
In this appendix we summarize the properties of the Calabi--Yau constructed above, in some tables. In all the tables we give a reference to where the K3 surface $S$ and the associated automorphism $\alpha_S$ (or $\beta_S$) are constructed, we list the properties of the fixed locus of the automorphism and of its powers (we follow the notation introduced before), we compute  the Hodge numbers and we say whenever the family of Calabi--Yau constructed does not admit maximal unipotent monodromy (this is denoted by No MUM) and whenever the family is ``new'', in the sense that it is not contained in the list \cite{list} of known Calabi--Yau 3-folds. We omit the cases $n = 2, 3$ since they were already analyzed in previous papers.

\subsection{Order 4}

In the following table we assume that the curve of highest genus, $g(D)$, fixed by $\alpha_S^2$ is also fixed by $\alpha_S$. The first line corresponds to the assumption that $\alpha^2$ fixes two elliptic curves. 
\newpage
{\tiny
\begin{center}
\begin{longtable}{|c|c||c|c|c|c|c|c||c|c||c|c|}
\caption{Hodge numbers of $X_4$ in case case 1) Proposition \ref{prop: Borcea Voisin 4}.}
\label{table: Hodge numbers case 1}\\
\hline
 & Ref K3 & $m$ & $r$ & $n_1$ & $k$ & $a$ & $g(D)$ & $h^{1, 1}(X)$ & $h^{2, 1}(X)$ & No MUM & new\\
\hline
\hline
1 & \cite[Table 1, l. 1]{AS2} & 6 & 6 & 4 & 1 & 0 & 1 & 25 & 13 & &\\
\hline
2 & \cite[Table 1, l. 2]{AS2} & 5 & 7 & 4 & 1 & 0 & 1 & 29 & 11 & &\\
\hline
3 & \cite[Table 1, l. 3]{AS2} & 4 & 10 & 6 & 2 & 0 & 1 & 46 & 10 & &\\
\hline
4 & \cite[Table 1, l. 4]{AS2} & 4 & 8 & 4 & 1 & 1 & 1 & 34 & 10 & &\\
\hline
5 & \cite[Table 1, l. 5]{AS2} & 3 & 9 & 4 & 1 & 2 & 1 & 39 & 9 & &\\
\hline
6 & \cite[Table 1, l. 6]{AS2} & 2 & 10 & 4 & 1 & 3 & 1 & 44 & 8 & &\\
\hline
7 & \cite[Table 2, l. 1]{AS2} & 7 & 1 & 0 & 1 & 0 & 3 & 9 & 27 & &\\
\hline
8 & \cite[Table 2, l. 2]{AS2} & 6 & 4 & 2 & 1 & 0 & 2 & 19 & 19 & &\\
\hline
9 & \cite[Table 2, l. 3]{AS2} & 6 & 2 & 0 & 1 & 1 & 3 & 14 & 26 & &\\
\hline
10 & \cite[Table 2, l. 4]{AS2} & 5 & 5 & 2 & 1 & 1 & 2 & 24 & 18 & &\\
\hline
11 & \cite[Table 2, l. 5]{AS2} & 4 & 6 & 2 & 1 & 2 & 2 & 29 & 17 & &\\
\hline
12 & \cite[Table 3, l. 1]{AS2} & 4 & 10 & 6 & 1 & 0 & 0 & 39 & 3 & X & \\
\hline
13 & \cite[Table 3, l. 2]{AS2} & 3 & 13 & 8 & 2 & 0 & 0 & 56 & 2 & X & \\
\hline
14 & \cite[Table 3, l. 3]{AS2} & 3 & 11 & 6 & 1 & 1 & 0 & 44 & 2 & X & X \\
\hline
15 & \cite[Table 3, l. 4]{AS2} & 2 & 16 & 10 & 3 & 0 & 0 & 73 & 1 & X & \\
\hline
16 & \cite[Table 3, l. 5]{AS2} & 2 & 14 & 8 & 2 & 1 & 0 & 61 & 1 & X & \\
\hline
17 & \cite[Table 3, l. 6]{AS2} & 2 & 12 & 6 & 1 & 2 & 0 & 49 & 1 & X & X \\
\hline
18 & \cite[Table 3, l. 7]{AS2} & 1 & 19 & 12 & 4 & 0 & 0 & 90 & 0 & X & \\
\hline
19 & \cite[Table 3, l. 8]{AS2} & 1 & 13 & 6 & 1 & 3 & 0 & 54 & 0 & X & X \\
\hline
\end{longtable}
\end{center}
}
In the following table we assume that $\alpha_S$ is an involution on the curve of highest degree, $g(D)$, fixed by $\alpha_S^2$. 
{\tiny
\begin{center}
\begin{longtable}{|c|c||c|c|c|c|c|c|c||c|c||c|c|}
\caption{Hodge numbers of $X_4$ in case 2) Proposition \ref{prop: Borcea Voisin 4}.}
\label{table: Hodge numbers case 2}\\
\hline
 & Ref. K3 & $m$ & $r$ & $n_1$ & $n_2$ & $k$ & $a$ & $g(D)$ & $h^{1, 1}(X)$ & $h^{2, 1}(X)$ & No MUM & new\\
\hline
\hline
1 & \cite[Table 5, l. 1]{AS2} & 10 & 2 & 2 & 2 & 0 & 0 & 10 & 17 & 29 & & \\
\hline
2 & \cite[Table 5, l. 2]{AS2} & 10 & 2 & 0 & 4 & 0 & 0 & 9 & 14 & 26 & & \\
\hline
3 & \cite[Table 5, l. 3]{AS2} & 8 & 6 & 2 & 4 & 1 & 0 & 7 & 32 & 20 & & \\
\hline
4 & \cite[Table 5, l. 4]{AS2} & 8 & 6 & 0 & 6 & 1 & 0 & 6 & 29 & 17 & & \\
\hline
5 & \cite[Table 5, l. 5]{AS2} & 6 & 10 & 6 & 2 & 2 & 0 & 6 & 53 & 17 & & \\
\hline
6 & \cite[Table 5, l. 6]{AS2} & 6 & 10 & 4 & 4 & 2 & 0 & 5 & 50 & 14 & & \\
\hline
7 & \cite[Table 5, l. 7]{AS2} & 6 & 10 & 2 & 6 & 2 & 0 & 4 & 47 & 11 & & \\
\hline
8 & \cite[Table 5, l. 8]{AS2} & 6 & 10 & 0 & 8 & 2 & 0 & 3 & 44 & 8 & & \\
\hline
9 & \cite[Table 5, l. 9]{AS2} & 4 & 14 & 6 & 4 & 3 & 0 & 3 & 68 & 8 & & \\
\hline
10 & \cite[Table 5, l. 10]{AS2} & 4 & 14 & 4 & 6 & 3 & 0 & 2 & 65 & 5 & & \\
\hline
11 & \cite[Table 5, l. 11]{AS2} & 2 & 18 & 10 & 2 & 4 & 0 & 2 & 89 & 5 & & \\
\hline
12 & \cite[Table 5, l. 12]{AS2} & 2 & 18 & 8 & 4 & 4 & 0 & 1 & 86 & 2 & & \\
\hline
\end{longtable}
\end{center}
}

In \cite[Table 6]{AS2} a list of admissible lattices associated to certain fixed loci is given. For some of them, the corresponding family of K3 surfaces is also constructed. In particular, all the cases listed in \cite[Table 6]{AS2} and with $g = 0$ are associated to a family of K3 surfaces, constructed in \cite[Example 7.2]{AS2}. If also $\alpha_S^2$ does not fix curves with positive genus, then we are in the assumption of Remark \ref{rem: no mum 4}. Here we list the Calabi--Yau of type $X_4$ corresponding to these assumptions.
{\tiny
\begin{center}
\begin{longtable}{|c|c||c|c|c|c|c|c|c||c|c||c|c|}
\caption{$\alpha$ fixes only isolated points, $\alpha^2$ fixes only rational curves.}
\label{table: only rational curves}\\
\hline
 & Ref. K3 & $m$ & $r$ & $n_1$ & $n_2$ & $k$ & $a$ & $g(D)$ & $h^{1, 1}(X)$ & $h^{2, 1}(X)$ & No MUM & new\\
\hline
\hline
1 & \cite[Table 6, l. 13]{AS2} & 5 & 7 & 2 & 2 & 0 & 0 & 0 & 22 & 4 & X & X\\
\hline
2 & \cite[Table 6, l. 18]{AS2} & 4 & 8 & 2 & 2 & 0 & 1 & 0 & 27 & 3 & X & X\\
\hline
3 & \cite[Table 6, l. 23]{AS2} & 3 & 9 & 2 & 2 & 0 & 2 & 0 & 32 & 2 & X & X \\
\hline
4 & \cite[Table 6, l. 27]{AS2} & 2 & 10 & 2 & 2 & 0 & 3 & 0 & 37 & 1 & X & X \\
\hline
5 & \cite[Table 6, l. 30]{AS2} & 1 & 11 & 2 & 2 & 0 & 4 & 0 & 42 & 0 & X & X \\
\hline
\end{longtable}
\end{center}
}

\begin{rem}
We observe that the Hodge numbers of the Calabi--Yau in Table \ref{table: Hodge numbers case 1}, line 11 are the same of the Calabi--Yau in Table \ref{table: Hodge numbers case 2}, line 4 and are mirrors of the ones of the Calabi--Yau in Table \ref{table: Hodge numbers case 2}, line 1. The Hodge numbers of the Calabi--Yau in Table \ref{table: Hodge numbers case 1}, line 9 are the same of the Calabi--Yau in Table \ref{table: Hodge numbers case 2}, line 2. The Hodge numbers of the Calabi--Yau in Table \ref{table: Hodge numbers case 1} line 8 are self mirror.
\end{rem}

\subsection{Order 6}

We compute the Hodge numbers of Calabi--Yau 3-folds of type $X_6$ as in Proposition \ref{prop: hodge numbers 6} if the K3 surface $S$ is one the the K3 surfaces listed in \cite[Table 1 (column 1--11)]{D}. The last line corresponds to the K3 surface constructed in Example \ref{example no mum X6}. We omit $g(D)$ since it is zero for any K3 considered in the Table.


{\tiny
\begin{center}
\begin{longtable}{|c||c|c|c|c|c||c|c|c||c|c|c||c|c||c|c||c|c|}
\caption{Hodge numbers of $X_6$.}
\label{table: Hodge numbers X6}\\
\hline
 & $n$ & $n'$ & $k$ & $a$ & $g(B)$ & $p_{(3, 4)}$ & $p_{(2, 5)}$ & $l$ & $N$ & $b$ & $g(F)$ & $r$ & $m$ & $h^{1, 1}$ & $h^{2, 1}$ & No MUM & new\\
\hline
\hline
1 & 0 & 0 & 2 & 0 & 5 & 12 & 0 & 1 & 2 & 0 & 10 & 2 & 10 & 29 & 29 & & \\
\hline
2 & 1 & 0 & 2 & 0 & 4 & 10 & 1 & 1 & 2 & 0 & 9 & 3 & 9 & 31 & 25 & & \\
\hline
3 & 2 & 0 & 2 & 0 & 3 & 8 & 2 & 1 & 2 & 0 & 8 & 4 & 8 & 33 & 21 & & \\
\hline
4 & 3 & 0 & 2 & 0 & 2 & 6 & 3 & 1 & 2 & 0 & 7 & 5 & 7 & 35 & 17 & & \\
\hline
5 & 3 & 1 & 3 & 0 & 3 & 10 & 1 & 1 & 3 & 0 & 7 & 6 & 7 & 43 & 19 & & \\
\hline
6 & 4 & 0 & 2 & 0 & 1 & 4 & 4 & 1 & 2 & 0 & 6 & 6 & 6 & 37 & 13 & & \\
\hline
7 & 4 & 1 & 3 & 0 & 2 & 8 & 2 & 1 & 3 & 0 & 6 & 7 & 6 & 45 & 15 & & \\
\hline
8 & 4 & 0 & 4 & 0 & 3 & 10 & 4 & 2 & 6 & 0 & 6 & 10 & 6 & 65 & 17 & & \\
\hline
9 & 5 & 0 & 2 & 0 & 0 & 2 & 5 & 1 & 2 & 0 & 5 & 7 & 5 & 39 & 9 & & \\
\hline
10 & 5 & 1 & 3 & 0 & 1 & 6 & 3 & 1 & 3 & 0 & 5 & 8 & 5 & 47 & 11 & & \\
\hline
11 & 5 & 0 & 4 & 0 & 2 & 8 & 5 & 2 & 6 & 0 & 5 & 11 & 5 & 67 & 13 & & \\
\hline
12 & 6 & 1 & 3 & 0 & 0 & 4 & 4 & 1 & 3 & 0 & 4 & 9 & 4 & 49 & 7 & & \\
\hline
13 & 6 & 0 & 4 & 0 & 1 & 6 & 6 & 2 & 6 & 0 & 4 & 12 & 4 & 69 & 9 & & \\
\hline
14 & 7 & 0 & 4 & 0 & 0 & 4 & 7 & 2 & 6 & 0 & 3 & 13 & 3 & 71 & 5 & & \\
\hline
15 & 7 & 1 & 5 & 0 & 1 & 8 & 5 & 2 & 7 & 0 & 3 & 14 & 3 & 72 & 7 & & \\
\hline
16 & 8 & 1 & 5 & 0 & 0 & 6 & 6 & 2 & 7 & 0 & 2 & 15 & 2 & 81 & 3 & & \\
\hline
17 & 8 & 0 & 6 & 0 & 1 & 8 & 8 & 3 & 10 & 0 & 2 & 18 & 2 & 101 & 5 & & \\
\hline
18 & 9 & 0 & 6 & 0 & 0 & 6 & 9 & 3 & 10 & 0 & 1 & 19 & 1 & 103 & 1 & & \\
\hline
\hline
19 & 6 & 1 & 3 & 0 & 0 & 4 & 4 & 1 & 5 & 0 & 0 & 11 & 2 & 55 & 1 & X & X \\
\hline
\end{longtable}
\end{center}
}

\bibliographystyle{amsplain} 

\end{document}